\documentclass[12pt]{article}
\usepackage{amssymb,amsthm,hyperref}
\newtheorem{thm}{Theorem}[section]

 \newtheorem{lem}{Lemma}[section]
 \newtheorem{prop}{Proposition}[section]
 \newtheorem{defn}{Definition}[section]%是否是全文计数？
\theoremstyle{remark}
\newtheorem{rem}{Remark}[section]

\title{Commutator estimates in Besov-Morrey spaces with applications
to the well-posedness of the Euler equations and ideal MHD system }
\author{Jiang Xu\thanks {E-mail: jiangxu\underline{ }79@yahoo.com.cn}\\
\small{\textit{Department of Mathematics}},\\\small{\textit{Nanjing
University of Aeronautics and Astronautics}},
\\ \small{\textit{Nanjing 211106, P.R.China}}\\[5mm]
Yong Zhou\thanks{E-mail: yzhoumath@zjnu.edu.cn}\\
\small{\textit{Department of Mathematics}},\\
\small{\textit{Zhejiang Normal University, Jinhua 321004,
P.R.China}}}
\date{}
\begin{document}
\maketitle{} \begin{abstract} We develop commutator estimates
in the framework of Besov-Morrey spaces, which are modeled on Besov
spaces and the underlying norm is of Morrey space rather than the
usual $L^{p}$ space. As direct applications of commutator estimates,
we establish the local well-posedness and blow-up criterion of
solutions in Besov-Morrey spaces for the incompressible
Euler equations and ideal MHD system. Main analysis tools are the
Littlewood-Paley decomposition and Bony's para-product formula.
\end{abstract}

\hspace{-0.5cm}\textbf{Keywords.} \small{Well-posedness;
 Euler equations; ideal MHD system; Besov-Morrey spaces}\\

\hspace{-0.5cm}\textbf{AMS subject classification:} \small{35L25;\
35L45;\ 76N15}

\section{Introduction}\label{sec:1}

\subsection{Euler equations and MHD system}
In this paper, the one interest is to consider the incompressible
Euler equations for perfect fluid
\begin{equation}
\left\{
\begin{array}{l}
\partial_{t}\upsilon + (\upsilon\cdot\nabla)\upsilon+\nabla P = 0,\ \ (x,t)\in \mathbb{R}^{n}\times(0,\infty),\\
\mathrm{div}\upsilon=0,\ \ (x,t)\in \mathbb{R}^{n}\times(0,\infty),\\
\upsilon(x,0)=\upsilon_{0}(x),\ \  x\in \mathbb{R}^{n},
\end{array} \right.\label{R-E1}
\end{equation}
where $n\geq2$,
$\upsilon=\upsilon(x,t)=(\upsilon^1,\upsilon^2,\cdot\cdot\cdot,\upsilon^n)$
stands for the velocity of the fluid, $P=P(x,t)$ is the pressure,
and $\upsilon_{0}(x)$ is the given initial velocity satisfying
$\mathrm{div}\upsilon_{0}=0$.

The other interest is to consider the ideal magneto-hydrodynamics
(MHD) system
\begin{equation}
\left\{
\begin{array}{l}
\partial_{t}\upsilon + (\upsilon\cdot\nabla)\upsilon-(b\cdot\nabla)b+\nabla \Pi = 0,\ \ (x,t)\in \mathbb{R}^{n}\times(0,\infty),\\
\partial_{t}b+(\upsilon\cdot\nabla)b-(b\cdot\nabla)\upsilon=0,\ \ (x,t)\in \mathbb{R}^{n}\times(0,\infty),\\
\mathrm{div}\upsilon=0,\ \ \mathrm{div}b=0,\ \ (x,t)\in \mathbb{R}^{n}\times(0,\infty),\\
\upsilon(x,0)=\upsilon_{0}(x),\ \ b(x,0)=b_{0}(x),\ \   x\in
\mathbb{R}^{n},
\end{array} \right.\label{R-E2}
\end{equation}
where $n\geq2$,
$\upsilon=\upsilon(x,t)=(\upsilon^1,\upsilon^2,\cdot\cdot\cdot,\upsilon^n)$
denotes the velocity of the fluid,
$b=b(x,t)=(b^1,b^2,\cdot\cdot\cdot,b^n)$ denotes the magnetic filed,
and $\Pi=P(x,t)+\frac{1}{2}|b(x,t)|^2$ is the total pressure.
$\upsilon_{0}(x)$ and $b_{0}$ are the initial velocity and initial
magnetic fields satisfying $\mathrm{div}\upsilon_{0}=0$ and
$\mathrm{div}b_{0}=0$, respectively.

\subsection{Related results}
For the well-posedness of the system (\ref{R-E1}), there are many
results available. Given $\upsilon_{0}\in H^{s},\ s>1+n/2$, Kato
\cite{K} established the local existence and uniqueness of regular
solution belonging to $C([0,T];H^{s}({\mathbb{R}^n}))$ with
$T=T(\|\upsilon_{0}\|_{H^{s}({\mathbb{R}^n})})$. Later, various
function spaces are used to consider the well-posedness for the
incompressible Euler equations. Kato and Ponce \cite{KP} extended
the result to the fractional-order Sobolev space $W^{s,p}$ with
$s>1+n/p,\ 1<p<\infty$. Vishik \cite{V1} showed the global
well-posedness  in the critical Besov space
$B^{1+2/p}_{p,1}({\mathbb{R}^2})(1<p<\infty)$. Subsequently, Vishik
\cite{V2} proved the existence ($n=2$) and uniqueness ($n\geq2$)
result for (\ref{R-E1}) with initial vorticity belonging to a space
of Besov type. In \cite{Z}, the second author generalized the
results of Vishik in critical Besov space
$B^{1+n/p}_{p,1}(1<p<\infty,\ n\geq3)$. Pak and Park \cite{PP}
considered the endpoint Besov space
$B^{1}_{\infty,1}({\mathbb{R}^n})$ and proved the corresponding
results. Chae \cite{C1} studied the case of the initial data
belonging to the Triebel-Lizorkin space. Based on \cite{C1}, Chen,
Miao and Zhang \cite{CMZ} studied the local well-posedness of the
ideal MHD system (\ref{R-E2}) in the Triebel-Lizorkin space. Miao
and Yuan \cite{MY} established the existence results for
(\ref{R-E2}) in the critical Besov space
$B^{1+n/p}_{p,1}({\mathbb{R}^n})(1\leq p\leq\infty)$.

For the blow-up criterion of solutions, Beale, Kato and Majda
\cite{BKM} showed a celebrated criterion for solutions in terms of
the vorticity $\omega=\nabla\times\upsilon$, namely,
  $\limsup_{t\rightarrow T_{*}}\|\upsilon(t)\|_{H^{s}}=\infty$
    if only if $\int^{T_{*}}_{0}\|\omega(t)\|_{L^\infty}dt=\infty.$
Subsequently, this result is extended to a larger class of solutions
by replacing the $L^\infty$ norm by the $BMO$ norm for the
vorticity, and $H^{s}({\mathbb{R}^n})$
 by $W^{s,p}({\mathbb{R}^n})$ for the velocity in the work \cite{KT} of Kozono and Taniuchi.
In their continuation work \cite{KOT}, Kozono, Ogawa and Taniuchi
further give a sharper criterion, which the the $BMO$ norm is
replace by the Besov space $\dot{B}^{0}_{\infty,\infty}$ for the
vorticity, since the continuous embedding $L^\infty\hookrightarrow
BMO\hookrightarrow\dot{B}^{0}_{\infty,\infty}$. By replacing
$W^{s,p}$ by $B^{s}_{p,r}$ for the velocity, Chae \cite{C1,C2}
establish the criterion by the $\dot{F}^{0}_{\infty,\infty}$ norm
for the vorticity. Chen, Miao and Zhang \cite{CMZ} also obtained a
similar criterion of (\ref{R-E2}) by the
$\dot{F}^{0}_{\infty,\infty}$ norm for $\omega$ and $\nabla\times
b$.

Recently, Based on the weak $L^{p}$ spaces $L^{p,\infty}$, R. Takada
\cite{TR} introduced Besov type function spaces
$B^{s,\infty}_{p,r}$, and established the local well-posedness of
solutions for (\ref{R-E1}) in the weak spaces. Following from this
line of study, we introduce a class of relatively weaker function
spaces, which first initialed by Kozono and Yamazaki \cite{KY} to
obtain the critical regularity for Navier-Stokes equations. They are
modeled on Besov spaces, but the underlying norm is of Morrey type
rather than $L^{p}$ or $L^{p,\infty}$ . One calls them Besov-Morrey
spaces. So far, there are few results constructed in the
Besov-Morrey spaces for systems (\ref{R-E1}) and (\ref{R-E2}). The
main aim of this paper is to answer these problems.

Before stating our main results, we first recall the definitions of
Morrey spaces and Besov-Morrey spaces (see \cite{KY}) for the
convenience of reader.

\begin{defn}\label{defn1.1}
For $1\leq q\leq p\leq\infty$, the Morrey space
$M^{p}_{q}(\mathbb{R}^{n})$ is defined as the set of functions
$f(x)\in L^{q}_{loc}(\mathbb{R}^{n})$ such that
$$\|f\|_{M^{p}_{q}(\mathbb{R}^{n})}:=\sup_{x_{0}\in\mathbb{R}^{n}}\sup_{r>0}r^{n/p-n/q}\Big(\int_{B(x_{0},r)}|f(y)|^{q}dy\Big)^{1/q}<\infty.$$
\end{defn}
Let us remark that it is easy to see that the relation
$M^{p}_{q_{1}}(\mathbb{R}^{n})\subset M^{p}_{q_{2}}(\mathbb{R}^{n})$
with $1\leq q_{2}<q_{1}\leq p\leq\infty,
M^{p}_{p}(\mathbb{R}^{n})=L^{p}$ and
$M^{\infty}_{q}(\mathbb{R}^{n})=L^{\infty}(\mathbb{R}^{n})$ for
$1\leq q\leq\infty$. In addition, from \cite{KY}, we know that
$L^{p,\infty}(\mathbb{R}^{n})\subset M^{p}_{q}(\mathbb{R}^{n})$ with
$1\leq q< p<\infty$.

Let $\chi(t)\in C^{\infty}_{0}(\mathbb{R})$ such that
$0\leq\chi(t)\leq1, \chi(t)\equiv1$ for $t\leq1$ and
$\mathrm{supp}\chi\subset[0,2]$. Denote
$\mathcal{F}[\varphi_{j}](\xi)=\chi(2^{-j}|\xi|)-\chi(2^{1-j}|\xi|)$
and $\mathcal{F}[\varphi_{(0)}](\xi)=\chi(|\xi|)$ for $\xi\in
\mathbb{R}^{n}$, where $\mathcal{F}[g]$ denotes the Fourier
transform of $g$ on $R^{n}$. Then
$$\mathcal{F}[\varphi_{(0)}](\xi)+\sum_{j\geq1}\mathcal{F}[\varphi_{j}](\xi)=1,\ \ \ \mbox{for}\ \ \xi\in \mathbb{R}^{n};$$
$$\sum_{j\in\mathbb{Z}}\mathcal{F}[\varphi_{j}](\xi)=1,\ \ \ \mbox{for}\ \ \xi\in \mathbb{R}^{n}\backslash\{0\}.$$
Given $f\in \mathcal{S}'$, where $\mathcal{S}'$ is the dual space of
Schwartz class $\mathcal{S}$. To define the homogeneous Besov-Morrey
spaces, we set
$$\dot{\Delta}_{j}f=\mathcal{F}^{-1}\varphi_{j}(\cdot)\mathcal{F}f,\ \ \ \ j=0,\pm1,\pm2,...$$

\begin{defn}\label{defn1.2}
For $1\leq q\leq p\leq\infty,\ 1\leq r\leq\infty,$ and $s\in
\mathbb{R}$, the homogeneous Besov-Morrey space
$\dot{N}^{s}_{p,q,r}$ is defined by
$$\dot{N}^{s}_{p,q,r}=\{f\in \mathcal{S}'/\mathcal{P}:\|f\|_{\dot{N}^{s}_{p,q,r}}<\infty\},$$
where
$$\|f\|_{\dot{N}^{s}_{p,q,r}}
=\cases{\Big(\sum_{j\in\mathbb{Z}}(2^{js}\|\dot{\Delta}_{j}f\|_{M^{p}_{q}})^{r}\Big)^{1/r},\
\ r<\infty, \cr \sup_{j\in\mathbb{Z}}
2^{js}\|\dot{\Delta}_{j}f\|_{M^{p}_{q}},\ \ r=\infty} $$ and
$\mathcal{P}$ denotes the set of polynomials with $n$ variables.
\end{defn}

To define the inhomogeneous Besov-Morrey spaces, we set
$$\Delta_{j}f=\cases{0,\, \, \, \, \, \, \, \, \hspace{24mm} \ j\leq-2,\cr
\mathcal{F}^{-1}[\varphi_{(0)}](\xi)\mathcal{F}[f],\ \ \ j=-1,\cr
\mathcal{F}^{-1}[\varphi_{j}](\xi)\mathcal{F}[f], \ \hspace{3mm} \
j=0,1,2,...}$$

\begin{defn}\label{defn2.2}
For $1\leq q\leq p\leq\infty,\ 1\leq r\leq\infty,$ and $s\in
\mathbb{R}$, the inhomogeneous Besov space $N^{s}_{p,q,r}$ is
defined by
$$N^{s}_{p,q,r}=\{f\in \mathcal{S}':\|f\|_{N^{s}_{p,q,r}}<\infty\},$$
where
$$\|f\|_{N^{s}_{p,q,r}}
=\cases{\Big(\sum_{j=-1}^{\infty}(2^{js}\|\Delta_{j}f\|_{M^{p}_{q}})^{r}\Big)^{1/r},\
\ r<\infty, \cr \sup_{j\geq-1} 2^{js}\|\Delta_{q}f\|_{M^{p}_{q}},\ \
r=\infty.}$$
\end{defn}
It is not difficult to see that the Besov-Morrey space
$\dot{N}^{s}_{p,q,r}$ and $N^{s}_{p,q,r}$ are the Banach spaces.
Recall that the standard homogeneous Besov space $\dot{B}^{s}_{p,r}$
and inhomogeneous Besov space $B^{s}_{p,r}$ (see, e.g., \cite{BCD}),
where the $L^{p}$ space is replaced by the Morrey space $M^{p}_{q}$
now.

Main results are stated as follows.
\begin{thm}\label{thm1.1}(1)(Local-time existence)
Let $1<q\leq p<\infty.$ Assume that $s$ and $r$ satisfy $s>1+n/p,\
1\leq r\leq\infty$\ or\ $s=1+n/p,\ r=1$. Suppose that the initial
data $\upsilon_{0}\in N^s_{p,q,r}$ satisfying
$\mathrm{div}\upsilon_{0}=0$. Then
there exist $T_{1}>0$ and a unique solution $\upsilon$ of (\ref{R-E1}) such that $\upsilon\in C([0,T_{1}],N^s_{p,q,r})$.\\
(2)(Blow-up criterion)
\begin{itemize}
\item[(i)] Let $s>1+n/p$ and\ $1\leq r\leq\infty$. Then the local-in-time solution $\upsilon\in C([0,T_{1}],N^s_{p,q,r})$ blows up at $T_{*}>T_{1}$ in $N^s_{p,q,r}$, namely
    $$\limsup_{t\rightarrow T_{*}}\|\upsilon(t)\|_{N^s_{p,q,r}}=\infty$$
    if only if $\int^{T_{*}}_{0}\|(\nabla\times\upsilon)(t)\|_{\dot{B}^{0}_{\infty,\infty}}dt=\infty;$
\item[(ii)] Let $s=1+n/p$ and \ $r=1$. Then the local-in-time solution $\upsilon\in C([0,T_{1}],N^{1+n/p}_{p,q,1})$ blows up at $T_{*}>T_{1}$ in $N^{1+n/p}_{p,q,1}$, namely
    $$\limsup_{t\rightarrow T_{*}}\|\upsilon(t)\|_{N^{1+n/p}_{p,q,1}}=\infty$$
    if only if
    $\int^{T_{*}}_{0}\|(\nabla\times\upsilon)(t)\|_{\dot{B}^{0}_{\infty,1}}dt=\infty$.
\end{itemize}
\end{thm}
\begin{rem}
Since $L^{p}(\mathbb{R}^{n})\subset
L^{p,\infty}(\mathbb{R}^{n})\subset M^{p}_{q}(\mathbb{R}^{n})$, we
have the continuous embeddings $B^{s}_{p,r}(\mathbb{R}^{n})\\
\hookrightarrow B^{s,\infty}_{p,r}(\mathbb{R}^{n})\hookrightarrow
N^s_{p,q,r}(\mathbb{R}^{n}).$ Therefore, the local existence result
contains the previous ones by Chae \cite{C2}, Takada \cite{TR} and
Zhou \cite{Z}. In addition, due to
$L^{\infty}(\mathbb{R}^{n})\hookrightarrow BMO(\mathbb{R}^{n})
\hookrightarrow
\dot{B}^{0}_{\infty,\infty}(\mathbb{R}^{n})=\dot{F}^{0}_{\infty,\infty}(\mathbb{R}^{n})$,
the blow-up criterion in Theorem \ref{thm1.1} can be regarded as an
improvement of the original Beale-Kato-Majda criterion \cite{BKM}
and a generalization of Chae \cite{C2}.
\end{rem}

For the MHD system (\ref{R-E2}), we have the similar result.
\begin{thm}\label{thm1.2}(1)(Local-time existence)
Let $1<q\leq p<\infty.$ Assume that $s$ and $r$ satisfy $s>1+n/p,\
1\leq r\leq\infty$\ or\ $s=1+n/p,\ r=1$. Suppose that the initial
data $(\upsilon_{0},b_{0})\in N^s_{p,q,r}$ satisfying
$\mathrm{div}\upsilon_{0}=0$ and $\mathrm{div}b_{0}=0$. Then
there exist $T_{2}>0$ and a unique solution $(\upsilon,b)$ of (\ref{R-E2}) such that $(\upsilon,b)\in C([0,T_{2}],N^s_{p,q,r})$.\\
(2)(Blow-up criterion)
\begin{itemize}
\item[(i)] Let $s>1+n/p$ and\ $1\leq r\leq\infty$. Then the local-in-time solution $(\upsilon,b)\in C([0,T_{1}],N^s_{p,q,r})$ blows up at $T_{*}>T_{2}$ in $N^s_{p,q,r}$, namely
    $$\limsup_{t\rightarrow T_{*}}\|(\upsilon,b)(t)\|_{N^s_{p,q,r}}=\infty$$
    if only if $\int^{T_{*}}_{0}(\|(\nabla\times\upsilon,\nabla\times b)(t)\|_{\dot{B}^{0}_{\infty,\infty}}dt=\infty;$
\item[(ii)] Let $s=1+n/p$ and \ $r=1$. Then the local-in-time solution $(\upsilon,b)\in C([0,T_{1}],N^{1+n/p}_{p,q,1})$ blows up at $T_{*}>T_{1}$ in $N^{1+n/p}_{p,q,1}$, namely
    $$\limsup_{t\rightarrow T_{*}}\|(\upsilon,b)(t)\|_{N^{1+n/p}_{p,q,1}}=\infty$$
    if only if $\int^{T_{*}}_{0}(\|(\nabla\times\upsilon,\nabla\times b)(t)\|_{\dot{B}^{0}_{\infty,1}}dt=\infty.$
\end{itemize}
\end{thm}

\begin{rem}
In the proof of Theorems \ref{thm1.1}-\ref{thm1.2}, inspired by
\cite{CMZ,TR}, we introduce the following particle trajectory
mappings
\begin{equation}
\left\{
\begin{array}{l}
\partial_{t}X(\alpha,t)=\upsilon(X(\alpha,t),t),\\
X(\alpha,0)=\alpha,
\end{array} \right.\label{R-E3}
\end{equation}
\begin{equation}
\left\{
\begin{array}{l}
\partial_{t}Y(\alpha,t)=(\upsilon-b)(Y(\alpha,t),t),\\
Y(\alpha,0)=\alpha,
\end{array} \right.\label{R-E4}
\end{equation}
and \begin{equation} \left\{
\begin{array}{l}
\partial_{t}Z(\alpha,t)=(\upsilon+b)(Z(\alpha,t),t),\\
Z(\alpha,0)=\alpha
\end{array} \right.\label{R-E5}
\end{equation}
to estimate the frequency-localization solutions to the Euler
equations (\ref{R-E1}) and MHD system (\ref{R-E2}) in the
$M^{p}_{q}(\mathbb{R}^{n})$ space, respectively. It is worth noting
that we handle with the coupling effect of the velocity field
$\upsilon(x,t)$ and the magnetic field $b(x,t)$ in (\ref{R-E2})
effectively by (\ref{R-E4})-(\ref{R-E5}).

Secondly, to deal with frequency-localized nonlinear terms, we
develop new commutator estimates in the weaker Besov-Morrey space
$\dot{N}^{s}_{p,q,r}$ by the Bony's para-product formula. In
addition, we should mention the recent preprint \cite{T}, where he
has announced the partial content (the sup-critical case) of Theorem
\ref{thm1.1}, however, the proof of local existence was not
available. In fact, by the careful investigation, we think the proof
is not obvious. More concretely speaking, according to the work
\cite{Z} by the second author, we adopt some revised approximate
iteration systems instead of that in \cite{C1} to construct the
local existence of solutions of (\ref{R-E1}) and (\ref{R-E2}).
\end{rem}

\begin{rem}
Theorems \ref{thm1.1}-\ref{thm1.2} can be regarded as the supplements on the local existence theory of the Euler equations and ideal MHD system. However,
the \textit{global} existence of solutions to the 2-dimension case in the framework of Besov-Morrey spaces still remains unsolvable, which is our next consideration.
\end{rem}

At the end of Introduction, we also mention other lines of recent
study for incompressible Euler equations and MHD system, such as
\cite{D,DF} and \cite{Z2,ZXF}, where they considered the local
well-posedness of density-dependent Euler equations and MHD system
in several space dimensions.

The rest of this paper unfolds as follows. In Section~\ref{sec:2},
we briefly review some basic properties of Besov-Morrey spaces. In
Section~\ref{sec:3}, we give some key lemmas. In particular, we
develop new estimates of commutator in Besov-Morrey spaces, which
play important roles in the proof of our main theorems.
Section~\ref{sec:4} is devoted to the total proof of Theorem
\ref{thm1.1}. Finally in Section~\ref{sec:5}, we prove the local
well-posedness of MHD system. For brevity, we give the approximate
linear system and crucial estimates only.

\section{Preliminary}\label{sec:2}
\setcounter{equation}{0} Throughout the paper, $f\lesssim g$ denotes
$f\leq Cg$, where $C>0$ is a generic constant. $f\thickapprox g$
means $f\lesssim g$ and $g\lesssim f$. In this section, we present
some properties in Besov-Morrey spaces defined in Sect.~\ref{sec:1}
by using the Littlewood-Paley dyadic decomposition. Indeed,
Besov-Morrey spaces share many of the properties of Besov spaces,
but they represent local oscillations and singularities of functions
more precisely. For more details, please refer to \cite{KY,M}.

First, we recall the Bernstein's inequality for $M^{p}_{q}$ as for
the case of $L^{p}$ spaces.
\begin{lem}\label{lem2.1}
Assume that $f\in M^{p}_{q}(\mathbb{R}^{n})$ with $1\leq q\leq
p\leq\infty$ and $\mathrm{supp}\mathcal{F}[f]\subset
\{2^{j-1}\leq|\xi|<2^{j+1}\},$ then there exists a constant $C_{k}$
such that the following inequalities holds:
$$C^{-1}_{k}2^{jk}\|f\|_{M^{p}_{q}}\leq\|D^{k}f\|_{M^{p}_{q}}\leq C_{k}2^{jk}\|f\|_{M^{p}_{q}}, \ \ \mbox{for all}\ \ k\in \mathbb{N},$$
where $\mathcal{F}[f]$ denotes the usual Fourier transform of $f$.
\end{lem}
In the Morrey space $M^{p}_{q}$, since
$$\|\phi\ast f\|_{M^{p}_{q}}\leq C\|\phi\|_{L^1}\|f\|_{M^{p}_{q}}$$
holds for $1\leq q\leq p\leq\infty$, we have the immediate
consequence of Lemma \ref{lem2.1}.
\begin{lem}\label{lem2.2}
For $s\in \mathbb{R}, 1\leq q\leq p\leq\infty, 1\leq r\leq\infty,
k\in \mathbb{N}$, and $\mathrm{supp}\mathcal{F}[f]\subset
\{2^{j-1}\leq|\xi|<2^{j+1}\},$ there exists a constant $C_{k}$ such
that the following inequality holds:
$$C^{-1}_{k}\|D^{k}f\|_{\dot{N}^{s}_{p,q,r}}\leq\|f\|_{\dot{N}^{s+k}_{p,q,r}}\leq C_{k}\|D^{k}f\|_{\dot{N}^{s}_{p,q,r}}.$$
\end{lem}

Next we investigate the relation between the homogeneous and
inhomogeneous Besov-Morrey spaces.
\begin{lem}\label{lem2.3}
For $s>0, 1\leq q\leq p\leq\infty, 1\leq r\leq\infty$, the following
relations hold:
$$N^{s}_{p,q,r}(\mathbb{R}^{n})=M^{p}_{q}(\mathbb{R}^{n})\cap\dot{N}^{s}_{p,q,r}(\mathbb{R}^{n}),$$
$$\|f\|_{N^{s}_{p,q,r}}\thicksim\|f\|_{M^{p}_{q}}+\|f\|_{\dot{N}^{s}_{p,q,r}}.$$
\end{lem}
The proof of Lemma \ref{lem2.3} is standard, see \cite{BL} for the
similar details. From \cite{KY}, we have the following Sobolev-type
embedding lemma.
\begin{lem}\label{lem2.4}
For $s>0, 1\leq q\leq p\leq\infty, 1\leq r\leq\infty$, then
$$\dot{N}^{s}_{p,q,r}(\mathbb{R}^{n})\hookrightarrow \dot{B}^{s-n/p}_{\infty,r}(\mathbb{R}^{n}),\ \ \ N^{s}_{p,q,r}(\mathbb{R}^{n})\hookrightarrow B^{s-n/p}_{\infty,r}(\mathbb{R}^{n}).$$
\end{lem}

Since the embedding relation $\dot{B}^{s}_{p,r}\hookrightarrow
L^{\infty}, B^{s}_{p,r}\hookrightarrow L^{\infty}$ hold for $s>n/p,
1\leq p,r\leq\infty,$ or $s=n/p, 1\leq p\leq\infty, r=1$, we obtain
the following conclusion from Lemma \ref{lem2.4}.
\begin{lem}\label{lem2.5}
Both spaces $\dot{N}^{s}_{p,q,r}(\mathbb{R}^{n})$ and
$N^{s}_{p,q,r}(\mathbb{R}^{n})$ are Banach algebras for $s>n/p,
1\leq q\leq p\leq\infty, r\in[1,\infty]$ or $s=n/p, 1\leq q\leq
p\leq\infty$ and $r=1$.
\end{lem}

Finally, we present the Bony's para-product formula to end up this
section. For simplicity, we state the homogeneous case only.
\begin{defn}\label{defn2.1}
Let $f,g $ be two temperate distributions. The product $f\cdot g$
has the Bony's decomposition formally:
$$f\cdot g=\dot{T}_{f}g+\dot{T}_{g}f+\dot{R}(f,g), $$
where $\dot{T}_{f}g$ is paraproduct of $g$ by $f$,
$$ \dot{T}_{f}g=\sum_{j'\leq j-2}\dot{\Delta}_{j'}f\dot{\Delta}_{j}g=\sum_{j\in\mathbb{Z}}\dot{S}_{j-1}f\dot{\Delta}_{j}g$$
and the remainder $\dot{R}(f,g)$ is denoted by
$$\dot{R}(f,g)=\sum_{|j-j'|\leq1}\dot{\Delta}_{j}f\dot{\Delta}_{j'}g.$$
\end{defn}

The para-product of two temperate distributions is always defined,
since the general term of the para-product is spectrally localized
in dyadic shells. However, the remainder may not be defined. Roughly
speaking, it is defined when $f$ and $g$ belong to functional spaces
whose sum of regularity index is positive.  The reader is referred
to \cite{BCD} for more details on the subject.

\section{Key lemmas}\label{sec:3}
\setcounter{equation}{0}
In this section, we will present some key lemmas, which are used to prove the main results.
The first one is related to the particle trajectory mapping.
\begin{lem}\label{lem3.1}
Assume that $f\in M^{p}_{q}(R^{n})$ for $1\leq q\leq p\leq\infty$.
If $X:\alpha\mapsto X(\alpha)$ is a volume-preserving
diffeomorphism, then
\begin{eqnarray}\|f(\alpha)\|_{M^{p}_{q}}=\|f(X(\alpha))\|_{M^{p}_{q}}.\label{R-E666}\end{eqnarray}
\end{lem}
\begin{proof}
For $x_{0}\in R^{n}$ and $r>0$, $\exists! y_{0}\in R^{n}$ s.t.
$x_{0}=X(y_{0})$. Then
\begin{eqnarray}
\int_{B(x_{0},r)}|f(\alpha)|^{q}d\alpha&=&\int_{X^{-1}(B(x_{0},r))}|f(X(\alpha))|^{q}\det(\nabla_{\alpha}X(\alpha))d\alpha\nonumber\\
&=&\int_{X^{-1}(B(x_{0},r))}|f(X(\alpha))|^{q}d\alpha\nonumber\\
&=&\int_{B(y_{0},r)}|f(X(\alpha))|^{q}d\alpha,\label{R-E667}
\end{eqnarray}
which implies
\begin{eqnarray}
\sup_{x_{0}\in
R^{n}}\sup_{r>0}r^{n/p-n/q}\int_{B(x_{0},r)}|f(\alpha)|^{q}d\alpha=\sup_{y_{0}\in
R^{n}}\sup_{r>0}r^{n/p-n/q}\int_{B(y_{0},r)}|f(X(\alpha))|^{q}d\alpha,\label{R-E668}
\end{eqnarray}
so (\ref{R-E666}) follows from Definition\ref{defn1.1} immediately.
\end{proof}

The next one concerns the logarithmic Besov-Morrey inequality, which is very useful to establish
the blow-up criterion in the super-critical case.
\begin{lem}\label{lem3.2}[\cite{T}]
Let $s>n/p$ with $1\leq q\leq p\leq\infty,\ 1\leq r\leq\infty.$
Assume $f\in N^s_{p,q,r}$, then there exists a constant $C$ such
that the following inequality holds:
\begin{eqnarray}
\|f\|_{L^\infty}\leq
C\Big(1+\|f\|_{\dot{B}^{0}_{\infty,\infty}}(\log^{+}\|f\|_{N^s_{p,q,r}}+1)\Big).
\label{R-E6}
\end{eqnarray}
\end{lem}

In order to estimate the bilinear terms, we need the following
Moser-type inequalities in Besov-Morrey spaces.
\begin{lem}\label{lem3.3}[\cite{T}]
Let $s>n/p$ with $1\leq q\leq p<\infty,\ 1\leq r\leq\infty$ or
$p=r=\infty.$ Then exists a constant $C$ such that the following
inequalities hold:
\begin{eqnarray}
\|fg\|_{\dot{N}^s_{p,q,r}}\leq
C\Big(\|f\|_{M^{p_{1}}_{q_{1}}}\|g\|_{\dot{N}^s_{p_{2},q_{2},r}}+\|g\|_{M^{p_{3}}_{q_{3}}}
\|f\|_{\dot{N}^s_{p_{4},q_{4},r}}\Big),\label{R-E7}
\end{eqnarray}
\begin{eqnarray}
\|fg\|_{N^s_{p,q,r}}\leq
C\Big(\|f\|_{M^{p_{1}}_{q_{1}}}\|g\|_{N^s_{p_{2},q_{2},r}}+\|g\|_{M^{p_{3}}_{q_{3}}}\|f\|_{N^s_{p_{4},q_{4},r}}\Big),\label{R-E8}
\end{eqnarray}
and
\begin{eqnarray}
\|fg\|_{\dot{N}^s_{p,q,r}}\leq
C\Big(\|f\|_{\dot{N}^{-\alpha}_{p_{1},q_{1},r_{1}}}\|g\|_{\dot{N}^{s+\alpha}_{p_{2},q_{2},r_{2}}}
+\|g\|_{\dot{N}^{-\alpha}_{p_{3},q_{3},r_{3}}}\|f\|_{\dot{N}^{s+\alpha}_{p_{4},q_{4},r_{4}}}\Big)\label{R-E9}
\end{eqnarray}
for $\alpha>0$, where $1\leq q_{1}\leq p_{1}\leq\infty$ and $1\leq
q_{3}\leq p_{3}\leq\infty$, such that
$$1/p=1/p_{1}+1/p_{2}=1/p_{3}+1/p_{4},\ 1/r=1/r_{1}+1/r_{2}=1/r_{3}+1/r_{4},$$ $$ 1/q\leq1/q_{1}+1/q_{2},\ \ \mbox{and}\ \ 1/q\leq1/q_{3}+1/q_{4}.$$
\end{lem}

The last one concerns the commutator estimates, which plays the important role in
the treatment of frequency-localized nonlinear terms.
\begin{lem}\label{lem3.4}
For $s>0$, $1\leq q\leq p<\infty$ and $ 1\leq r\leq\infty$, there is
a constant $C$ such that
\begin{eqnarray}
\|2^{sj}\|[v\cdot\nabla,\dot{\Delta}_{j}]\theta\|_{M^{p}_{q}}\|_{\ell^{r}}
\leq C\Big(\|\nabla
v\|_{L^\infty}\|\theta\|_{\dot{N}^s_{p,q,r}}+\|\nabla\theta\|_{M^{p_{1}}_{q_{1}}}\|v\|_{N^s_{p_{2},q_{2},r}}\Big)\label{R-E10}
\end{eqnarray}
holds for all $\theta\in \dot{N}^s_{p,q,r}$ with $\nabla\theta\in
M^{p_{1}}_{q_{1}}$ and all $v\in N^s_{p_{2},q_{2},r}$ with $\nabla
v\in L^\infty$ such that $\mathrm{div}v=0$, where $1\leq q_{1}\leq
p_{1}\leq\infty$ such that $1/p=1/p_{1}+1/p_{2}$ and
$1/q\leq1/q_{1}+1/q_{2}$.
\end{lem}

\begin{proof}
By Bony's para-product, we decompose
$[v\cdot\nabla,\dot{\Delta}_{j}]\theta=K_{1}+K_{2}+K_{3}+K_{4}+K_{5}$
with
$$K_{1}=[T_{\upsilon^{i}}\partial_{i},\dot{\Delta}_{j}]\theta,\ \ K_{2}=-\dot{\Delta}_{j}T_{\partial_{i}\theta}\upsilon^{i}, \ \ K_{3}=T_{\partial_{i}\dot{\Delta}_{j}\theta}\upsilon^{i},$$
$$K_{4}=-\dot{\Delta}_{j}R(\upsilon^{i},\partial_{i}\theta),\ \ K_{5}=R(\upsilon^{i},\partial_{i}\dot{\Delta}_{j}\theta),$$
where the Einstein notation was used for simplicity.

From the definition of $\dot{\Delta}_{j}$, we have the almost
orthogonal properties:
$$\dot{\Delta}_{i}\Delta_{j}f\equiv 0 \ \ \ \mbox{if}\ \ \ |i-j|\geq 2,$$
$$\dot{\Delta}_{j}(\dot{S}_{j-1}f\dot{\Delta}_{i}g)\equiv 0\ \ \ \mbox{if}\ \ \ |i-j|\geq 5.$$
For $K_{1}$, it follows from the fact
$\mathrm{div}\dot{S}_{j-1}\upsilon=0$ for all $j\in\mathbb{Z}$ and
orthogonal properties that
\begin{eqnarray*}
K_{1}&=&T_{\upsilon^{i}}\partial_{i}\dot{\Delta}_{j}\theta-\dot{\Delta}_{j}T_{\upsilon^{i}}\partial_{i}\theta\nonumber\\&=&
\sum_{j'\in
\mathbb{Z}}\Big\{S_{j'-1}\upsilon^{i}\dot{\Delta}_{j}(\partial_{i}\dot{\Delta}_{j'}\theta)-\dot{\Delta}_{j}(S_{j'-1}\upsilon^{i}\dot{\Delta}_{j'}\partial_{i}\theta)\Big\}
\nonumber\\&=&
\sum_{|j-j'|\leq4}2^{jn}\int_{\mathbb{R}^n}\varphi_{0}(2^{j}(x-y))\Big\{S_{j'-1}\upsilon^{i}(x)-S_{j'-1}\upsilon^{i}(y)\Big\}\partial_{i}\dot{\Delta}_{j'}\theta(y)dy
\nonumber\\&=&
\sum_{|j-j'|\leq4}2^{j(n+1)}\int_{\mathbb{R}^n}\partial_{i}\varphi_{0}(2^{j}(x-y))\Big\{S_{j'-1}\upsilon^{i}(x)-S_{j'-1}\upsilon^{i}(y)\Big\}\dot{\Delta}_{j'}\theta(y)dy
\nonumber\\&=&
\sum_{|j-j'|\leq4}2^{j(n+1)}\int_{\mathbb{R}^n}\partial_{i}\varphi_{0}(2^{j}(x-y))\int_{0}^{1}((x-y)\cdot\nabla)S_{j'-1}\upsilon^{i}(x+\tau(y-x))d\tau\dot{\Delta}_{j'}\theta(y)dy
\nonumber\\&=&\sum_{|j-j'|\leq4}\int_{\mathbb{R}^n}\partial_{i}\varphi_{0}(z)\int_{0}^{1}(z\cdot\nabla)S_{j'-1}\upsilon^{i}(x-\tau2^{-j}z)d\tau\dot{\Delta}_{j'}\theta(x-2^{-j}z)dz,
\end{eqnarray*}
where we have performed the integration by parts. Then, by Young's
inequality in Morrey spaces, we obtain
\begin{eqnarray}
\|K_{1}\|_{M^p_{q}}&\leq& C\|\nabla
\upsilon\|_{L^\infty}\sum_{|j-j'|\leq4}\|\int_{\mathbb{R}^n}|z\nabla\varphi(z)||\dot{\Delta}_{j'}\theta(\cdot-2^{-j}z)dz|\|_{M^p_{q}}
\nonumber\\&\leq& C\|\nabla
\upsilon\|_{L^\infty}\sum_{|j-j'|\leq4}\|\dot{\Delta}_{j'}\theta\|_{M^p_{q}}.\label{R-E11}
\end{eqnarray}
For $K_{2}$, it is clear that
\begin{eqnarray*}
K_{2}=-\sum_{|j-j'|\leq4}\dot{\Delta}_{j}\Big\{(S_{j'-1}\partial_{i}\theta)(\dot{\Delta}_{j'}v^{i})\Big\}
\end{eqnarray*}
which implies that
\begin{eqnarray}
\|K_{2}\|_{M^p_{q}}&\leq&C\sum_{|j-j'|\leq4}\|S_{j'-1}\partial_{i}\theta\|_{M^{p_{1}}_{q_{1}}}\|\dot{\Delta}_{j'}v^{i}\|_{M^{p_{2}}_{q_{2}}}
\nonumber\\&\leq&
C\|\nabla\theta\|_{M^{p_{1}}_{q_{1}}}\sum_{|j-j'|\leq4}\|\dot{\Delta}_{j'}v\|_{M^{p_{2}}_{q_{2}}},\label{R-E12}
\end{eqnarray}
where we used the H\"{o}lder inequality in Morrey spaces with
$1/p=1/p_{1}+1/p_{2}, 1/q\leq1/q_{1}+1/q_{2}$ and $1\leq q_{1}\leq
p_{1}\leq\infty$.

For $K_{3}$, note that
$S_{j'-1}\partial_{i}\dot{\Delta}_{j}\theta=0$, if $j'\leq j$, we
may rewrite
\begin{eqnarray*}
K_{3}=\sum_{j'\geq
j+1}(S_{j'-1}\partial_{i}\dot{\Delta}_{j}\theta)(\dot{\Delta}_{j'}v^{i}),
\end{eqnarray*}
Then the H\"{o}lder inequality gives
\begin{eqnarray}
\|K_{3}\|_{M^p_{q}}&\leq&\sum_{j'\geq
j+1}\|S_{j'-1}\partial_{i}\dot{\Delta}_{j}\theta\|_{M^{p_{1}}_{q_{1}}}\|\dot{\Delta}_{j'}v^{i}\|_{M^{p_{2}}_{q_{2}}}
\nonumber\\&\leq& C\|\nabla\theta\|_{M^{p_{1}}_{q_{1}}}\sum_{j'\geq
j+1}\|\dot{\Delta}_{j'}v\|_{M^{p_{2}}_{q_{2}}}.\label{R-E13}
\end{eqnarray}
For $K_{4}$, by the definition of $\dot{R}$, we may rewrite
\begin{eqnarray*}
K_{4}&=&-\dot{\Delta}_{j}\Big\{\sum_{j'\in\mathbb{Z}}\sum_{|j-j'|\leq1}(\dot{\Delta}_{j'}v^{i})(\dot{\Delta}_{j''}\partial_{i}\theta)\Big\}\nonumber\\&=&
-\sum_{\max(j',j'')\geq
j-2}\sum_{|j-j'|\leq1}\dot{\Delta}_{j}\{(\dot{\Delta}_{j'}v^{i})(\dot{\Delta}_{j''}\partial_{i}\theta)\},
\end{eqnarray*}
which yields
\begin{eqnarray}
\|K_{4}\|_{M^p_{q}}&\leq&\sum_{\max(j',j'')\geq
j-2}\sum_{|j-j'|\leq1}\|\dot{\Delta}_{j''}\partial_{i}\theta\|_{M^{p_{1}}_{q_{1}}}\|\dot{\Delta}_{j'}v^{i}\|_{M^{p_{2}}_{q_{2}}}\nonumber\\&\leq&
C \|\nabla\theta\|_{M^{p_{1}}_{q_{1}}}\sum_{j'\geq
j-2}\|\dot{\Delta}_{j'}v\|_{M^{p_{2}}_{q_{2}}}.\label{R-E14}
\end{eqnarray}
For the last term $K_{5}$, it holds that
\begin{eqnarray*}
K_{5}&=&\sum_{j'\in\mathbb{Z}}\sum_{|j-j'|\leq1}(\dot{\Delta}_{j'}v^{i})(\dot{\Delta}_{j''}\partial_{i}\dot{\Delta}_{j}\theta)\Big\}\nonumber\\&=&
\sum_{|j-j'|\leq2}\sum_{|j-j'|\leq1}\{(\dot{\Delta}_{j'}v^{i})(\dot{\Delta}_{j''}\partial_{i}\dot{\Delta}_{j}\theta)\}.
\end{eqnarray*}
Then, we get
\begin{eqnarray}
\|K_{5}\|_{M^p_{q}}&\leq&\sum_{|j-j'|\leq2}\sum_{|j-j'|\leq1}\|\dot{\Delta}_{j''}\partial_{i}\dot{\Delta}_{j}\theta\|_{M^{p_{1}}_{q_{1}}}\|\dot{\Delta}_{j'}v^{i}\|_{M^{p_{2}}_{q_{2}}}
\nonumber\\&\leq& C\|\nabla\theta\|_{M^{p_{1}}_{q_{1}}}\sum_{j'\geq
j-2}\|\dot{\Delta}_{j'}v\|_{M^{p_{2}}_{q_{2}}}.\label{R-E15}
\end{eqnarray}
Together with these estimates (\ref{R-E11})-(\ref{R-E15}), we
conclude that
\begin{eqnarray}
\|[v\cdot\nabla,\dot{\Delta}_{j}]\theta\|_{M^p_{q}}&\leq& C\|\nabla
v\|_{L^\infty}\sum_{|j-j'|\leq4}\|\dot{\Delta}_{j'}\theta\|_{M^p_{q}}+C\|\nabla\theta\|_{M^{p_{1}}_{q_{1}}}\sum_{|j-j'|\leq4}\|\dot{\Delta}_{j'}v\|_{M^{p_{2}}_{q_{2}}}
\nonumber\\&&+C|\nabla\theta\|_{M^{p_{1}}_{q_{1}}}\sum_{j'\geq
j-2}\|\dot{\Delta}_{j'}v\|_{M^{p_{2}}_{q_{2}}}.\label{R-E16}
\end{eqnarray}
Finally, we apply the Young inequality for sequences to get
(\ref{R-E10}) immediately, where $s>0$ is required. Thus we complete
the proof of Lemma \ref{lem3.4}.
\end{proof}

\begin{lem}\label{lem3.5}
For $s>-1$, $1\leq q\leq p<\infty$ and $ 1\leq r\leq\infty$, there
is a constant $C$ such that
\begin{eqnarray}
\|2^{js}\|[v\cdot\nabla,\dot{\Delta}_{j}]\theta\|_{M^{p}_{q}}\|_{\ell^{r}}\leq
C\Big(\|\nabla
v\|_{L^\infty}\|\theta\|_{\dot{N}^s_{p,q,r}}+\|\theta\|_{M^{p_{1}}_{q_{1}}}\|v\|_{\dot{N}^{s+1}_{p_{2},q_{2},r}}\Big)\label{R-E17}
\end{eqnarray}
holds for all $\theta\in \dot{N}^s_{p,q,r}\cap M^{p_{1}}_{q_{1}}$
and all $v\in \dot{N}^{s+1}_{p_{2},q_{2},r}$ with $\nabla v\in
L^\infty$ such that $\mathrm{div}v=0$, where $1\leq q_{1}\leq
p_{1}\leq\infty$ such that $1/p=1/p_{1}+1/p_{2},
1/q\leq1/q_{1}+1/q_{2}$.
\end{lem}

\begin{proof}
As in the proof of Lemma \ref{lem3.4}, we decompose
$$[v\cdot\nabla,\dot{\Delta}_{j}]\theta:=K_{1}+K_{2}+K_{3}+K_{4}+K_{5},$$
The estimate of $K_{1}$ is still valid, however, different estimates
and needed for $K_{2},K_{3},K_{4}$ and $K_{5}$.

For $K_{2}$, since $\mathrm{div}\dot{\Delta}_{j'}v=0$\ for all
$j'\in \mathbb{Z}$, by integration by parts, we have
\begin{eqnarray*}
K_{2}&=&-\sum_{|j-j'|\leq4}\dot{\Delta}_{j}\Big\{(S_{j'-1}\partial_{i}\theta)(\dot{\Delta}_{j'}v^{i})\Big\}\nonumber\\
&=&-\sum_{|j-j'|\leq4}2^{jn}\int_{R^{n}}\varphi_{0}(2^{j}(x-y))(S_{j'-1}\partial_{i}\theta)(y)\dot{\Delta}_{j'}v^{i}(y)dy
\nonumber\\&=&-\sum_{|j-j'|\leq4}2^{j}2^{jn}\int_{R^{n}}\partial_{i}\varphi_{0}(2^{j}(x-y))(S_{j'-1}\theta)(y)\dot{\Delta}_{j'}v^{i}(y)dy
\nonumber\\&=&-\sum_{|j-j'|\leq4}2^{j}\int_{R^{n}}\partial_{i}\varphi_{0}(z)S_{j'-1}\theta(x-2^{-j}z)\dot{\Delta}_{j'}v^{i}(x-2^{-j}z)dz.
\end{eqnarray*}
Then by Young's and H\"{o}lder inequalities in Morrey spaces, we
obtain
\begin{eqnarray}
\|K_{2}\|_{M^p_{q}}&\leq&\sum_{|j-j'|\leq4}2^{j}\Big\|\int_{R^{n}}\partial_{i}\varphi_{0}(z)S_{j'-1}\theta(\cdot-2^{-j}z)\dot{\Delta}_{j'}v^{i}(\cdot-2^{-j}z)dz\Big\|_{M^p_{q}}
\nonumber\\&\leq&C\sum_{|j-j'|\leq4}2^{j}\|(S_{j'-1}\theta)(\dot{\Delta}_{j'}v^{i})\|_{M^p_{q}}\nonumber\\&\leq&C
\|\theta\|_{M^{p_{1}}_{q_{1}}}\sum_{|j-j'|\leq4}2^{j}\|\dot{\Delta}_{j'}v^{i}\|_{M^{p_{2}}_{q_{2}}}.\label{R-E18}
\end{eqnarray}
Using the Bernstein's inequality, we proceed $K_{3}$ as follows:
\begin{eqnarray}
\|K_{3}\|_{M^p_{q}}&\leq&\sum_{j'\geq
j+1}\|S_{j'-1}\partial_{i}\dot{\Delta}_{j}\theta\|_{M^{p_{1}}_{q_{1}}}\|\dot{\Delta}_{j'}v^{i}\|_{M^{p_{2}}_{q_{2}}}
\nonumber\\&\leq&C\sum_{j'\geq
j+1}2^{j}\|\dot{\Delta}_{j}\theta\|_{M^{p_{1}}_{q_{1}}}\|\dot{\Delta}_{j'}v^{i}\|_{M^{p_{2}}_{q_{2}}}
\nonumber\\&\leq&C\|\theta\|_{M^{p_{1}}_{q_{1}}}\sum_{j'\geq
j+1}2^{j}\|\dot{\Delta}_{j'}v^{i}\|_{M^{p_{2}}_{q_{2}}}.\label{R-E20}
\end{eqnarray}
For $K_{4}$, we have by integration by parts
\begin{eqnarray*}
K_{4}&=&-\sum_{\max(j',j'')\geq
j-2}\sum_{|j'-j''|\leq1}2^{jn}\int_{R^{n}}\varphi_{0}(2^{j}(x-y))(\dot{\Delta}_{j'}v^{i})(y)(\dot{\Delta}_{j''}\partial_{i}\theta)(y)dy
\nonumber\\&=&-\sum_{\max(j',j'')\geq
j-2}\sum_{|j'-j''|\leq1}2^{j}2^{jn}\int_{R^{n}}\partial_{i}\varphi_{0}(2^{j}(x-y))(\dot{\Delta}_{j'}v^{i})(y)
(\dot{\Delta}_{j''}\theta)(y)dy\nonumber\\&=&-\sum_{\max(j',j'')\geq
j-2}\sum_{|j'-j''|\leq1}2^{j}\int_{R^{n}}\partial_{i}\varphi_{0}(z)
(\dot{\Delta}_{j'}v^{i})(x-2^{-j}z)(\dot{\Delta}_{j''}\theta)(x-2^{-j}z)dz,
\end{eqnarray*}
which leads to
\begin{eqnarray}
\|K_{4}\|_{M^p_{q}}&\leq&\sum_{\max(j',j'')\geq
j-2}\sum_{|j'-j''|\leq1}2^{j}\|(\dot{\Delta}_{j'}v^{i})(\dot{\Delta}_{j''}\theta)\|_{M^p_{q}}
\nonumber\\&\leq&C\|\theta\|_{M^{p_{1}}_{q_{1}}}\sum_{j'\geq
j-2}2^{j}\|\dot{\Delta}_{j'}v\|_{M^{p_{2}}_{q_{2}}}.\label{R-E21}
\end{eqnarray}
For the estimate of $K_{5}$, we recall
$$K_{5}=\sum_{|j-j'|\leq2}\sum_{|j-j'|\leq1}\{(\dot{\Delta}_{j'}v^{i})(\dot{\Delta}_{j''}\partial_{i}\dot{\Delta}_{j}\theta)\}.$$
furthermore, we get
\begin{eqnarray}
\|K_{5}\|_{M^p_{q}}&\leq&\sum_{|j-j'|\leq2}\sum_{|j-j'|\leq1}\|\dot{\Delta}_{j''}\partial_{i}\dot{\Delta}_{j}\theta\|_{M^{p_{1}}_{q_{1}}}\|\dot{\Delta}_{j'}v^{i}\|_{M^{p_{2}}_{q_{2}}}
\nonumber\\&\leq&C\sum_{|j-j'|\leq2}2^{j}\|\dot{\Delta}_{j}\theta\|_{M^{p_{1}}_{q_{1}}}\|\dot{\Delta}_{j'}v^{i}\|_{M^{p_{2}}_{q_{2}}}
\nonumber\\&\leq&C\|\theta\|_{M^{p_{1}}_{q_{1}}}\sum_{|j-j'|\leq2}2^{j}\|\dot{\Delta}_{j'}v\|_{M^{p_{2}}_{q_{2}}}.\label{R-E22}
\end{eqnarray}
From these inequalities (\ref{R-E11}), (\ref{R-E18})-(\ref{R-E22}),
we are led to
\begin{eqnarray}
\|[v\cdot\nabla,\dot{\Delta}_{j}]\theta\|_{M^p_{q}}&\leq& C\|\nabla
v\|_{L^\infty}\sum_{|j-j'|\leq4}\|\dot{\Delta}_{j'}\theta\|_{M^p_{q}}+C\|\theta\|_{M^{p_{1}}_{q_{1}}}\sum_{|j-j'|\leq4}2^{j}\|\dot{\Delta}_{j'}v\|_{M^{p_{2}}_{q_{2}}}
\nonumber\\&&+C|\theta\|_{M^{p_{1}}_{q_{1}}}\sum_{j'\geq
j-2}2^{j}\|\dot{\Delta}_{j'}v\|_{M^{p_{2}}_{q_{2}}},\label{R-E23}
\end{eqnarray}
which implies that
\begin{eqnarray}
\|2^{js}\|[v\cdot\nabla,\dot{\Delta}_{j}]\theta\|_{M^{p}_{q}}\|_{\ell^{r}}\leq
C\Big(\|\nabla
v\|_{L^\infty}\|\theta\|_{\dot{N}^s_{p,q,r}}+\|\theta\|_{M^{p_{1}}_{q_{1}}}\|v\|_{\dot{N}^{s+1}_{p_{2},q_{2},r}}\Big)
\label{R-E24}
\end{eqnarray}
where $s+1>0$ is required. This just the inequality (\ref{R-E17}).
Hence the proof of Lemma \ref{lem3.5} is finished.
\end{proof}

\section{Proof of Theorem \ref{thm1.1}}
\label{sec:4} \setcounter{equation}{0}

 In this section, we begin to
prove Theorem \ref{thm1.1} with the aid of key Lemmas
\ref{lem3.1}-\ref{lem3.5}. The proof is divided into several steps,
since it is a bit longer.\\

Step 1: \textit{\underline{The linear equation of (\ref{R-E1})}}

Consider the linear transport system as in \cite{Z}:
\begin{equation}
\left\{
\begin{array}{l}
\partial_{t}\upsilon+(w\cdot\nabla)\upsilon+\nabla P = 0,\\
\mathrm{div}\upsilon=0,\\
\upsilon(x,0)=\upsilon_{0}(x).
\end{array} \right.\label{R-E25}
\end{equation}
Then we have following local existence result for (\ref{R-E25}),
which will be proved in the last step.
\begin{prop}\label{prop4.1}
Assume that $\mathrm{div}w=0, w\in L^{\infty}(0,T,N^{s}_{p,q,r})$
for some $T>0, s>1+n/p, 1<q\leq p<\infty, r\in[1,\infty]$ or
$s=1+n/p, 1<q\leq p<\infty$ and $r=1$. Then for any $\upsilon_{0}\in
N^{s}_{p,q,r}$ satisfying $\mathrm{div}\upsilon_{0}=0$, there exists
a unique solution $\upsilon\in C([0,T];N^{s}_{p,q,r})$ to the linear
system (\ref{R-E25}). And consequently, $\nabla P$ can be determined
uniquely.
\end{prop}

Step 2: \textit{\underline{Approximate solutions and uniform
estimates}}

The proof of main theorem depends on the standard iteration
argument. To obtain the approximate solutions, we first set
$\upsilon^{0}=0$ and then define $\{\upsilon^{m+1}\}$ as the
solutions of the following linear system
\begin{equation}
\left\{
\begin{array}{l}
\partial_{t}\upsilon^{m+1}+(\upsilon^{m}\cdot\nabla)\upsilon^{m+1}+\nabla P^{m+1} = 0,\\
\mathrm{div}\upsilon^{m+1}=0,\ \ \mathrm{div}\upsilon^{m}=0,\\
\upsilon^{m+1}(x,0)=S_{m+1}\upsilon_{0}(x),
\end{array} \right.\label{R-E26}
\end{equation}
for $m=0,1,2,...$. In \cite{C1}, Chae took a similar (not same)
iterative system to construct the local solution. But unfortunately,
the linear system (3.32)-(3.33) on p.671 of \cite{C1} is unsolvable,
since the system itself lacks consistence.

If we have the uniform estimate for the sequence $\{\upsilon^{m}\}$
by induction, which satisfies the assumptions in Proposition
\ref{prop4.1}, then the system (\ref{R-E26}) can be solved with
solution $\{\upsilon^{m+1}\}$.

For that purpose, we turn to derive the uniform estimates of
solutions. Applying the homogeneous operator
$\dot{\Delta}_{j}(j\in\mathbb{Z})$ to the first equation of
(\ref{R-E26}), we have
\begin{equation}
\partial_{t}\dot{\Delta}_{j}\upsilon^{m+1}+(\upsilon^{m}\cdot\nabla)\dot{\Delta}_{j}\upsilon^{m+1}
=[\upsilon^{m}\cdot\nabla,\dot{\Delta}_{j}]\upsilon^{m+1}-\dot{\Delta}_{j}\nabla
P^{m+1}.\label{R-E27}
\end{equation}
Define by $\{X^{m}(\alpha,t)\}$ the family of particle trajectory
mapping as follows
\begin{equation}
\left\{
\begin{array}{l}
\partial_{t}X^{m}(\alpha,t)=\upsilon^{m}(X^{m}(\alpha,t),t),\\
X^{m}(\alpha,0)=\alpha.
\end{array} \right.\label{R-E28}
\end{equation}
Note that $\mathrm{div}\upsilon^{m}=0$ implies that each
$\alpha\mapsto X^{m}(\alpha,t)$ is a volume-preserving mapping for
all $t>0$. It follows from the particle trajectory mapping that
\begin{equation}
\partial_{t}\dot{\Delta}_{j}\upsilon^{m+1}+(\upsilon^{m}\cdot\nabla)\dot{\Delta}_{j}\upsilon^{m+1}\Big|_{(x,t)
=(X^{m}(\alpha,t),t)}=\frac{\partial}{\partial{t}}\dot{\Delta}_{j}\upsilon^{m+1}(X^{m}(\alpha,t),t)\label{R-E29}
\end{equation}
which gives
\begin{eqnarray}
|\dot{\Delta}_{j}\upsilon^{m+1}(X^{m}(\alpha,t),t)|&\leq&|\dot{\Delta}_{j}\upsilon^{m+1}(\alpha,0)|+\int^{t}_{0}|\dot{\Delta}_{j}\nabla
P^{m+1}(X^{m}(\alpha,\tau),\tau)|d\tau
\nonumber\\&&+\int^{t}_{0}|[\upsilon^{m}\cdot\nabla,\dot{\Delta}_{j}]\upsilon^{m+1}(X^{m}(\alpha,\tau),\tau)|d\tau.\label{R-E30}
\end{eqnarray}
Taking the $M^{p}_{q}$ norm ($1\leq q<p<\infty$) on both sides of
(\ref{R-E30}), with the help of Lemma \ref{lem3.1}, we get
\begin{eqnarray}
\|\dot{\Delta}_{j}\upsilon^{m+1}(t)\|_{M^{p}_{q}}&\leq&
\|\dot{\Delta}_{j}\upsilon^{m+1}_{0}\|_{M^{p}_{q}}+C\int^{t}_{0}\|\dot{\Delta}_{j}\nabla
P^{m+1}(\tau)\|_{M^{p}_{q}}d\tau
\nonumber\\&&+C\int^{t}_{0}\|[\upsilon^{m}\cdot\nabla,\dot{\Delta}_{j}]\upsilon^{m+1}(\tau)\|_{M^{p}_{q}}d\tau.\label{R-E31}
\end{eqnarray}
Then, we multiply both sides by $2^{js}$ and take the $\ell^{r}$
norm, and use Minkowski's inequality to obtain
\begin{eqnarray}
\|\upsilon^{m+1}(t)\|_{\dot{N}^s_{p,q,r}}&\leq&
\|\upsilon^{m+1}_{0}\|_{\dot{N}^s_{p,q,r}}+C\int^{t}_{0}\|\nabla
P^{m+1}(\tau)\|_{\dot{N}^s_{p,q,r}}d\tau
\nonumber\\&&+C\int^{t}_{0}\Big\|2^{js}\|[\upsilon^{m}\cdot\nabla,\dot{\Delta}_{j}]\upsilon^{m+1}(\tau)\|_{M^{p}_{q}}\Big\|_{\ell^{r}}d\tau.\label{R-E32}
\end{eqnarray}

Thanks to the commutator estimate in Lemma \ref{lem3.4}, by taking
$p_{1}=\infty$ and $q_{1}=q_{2}=q$, we have
\begin{eqnarray}
&&\|2^{js}\|[\upsilon^{m}\cdot\nabla,\dot{\Delta}_{j}]\upsilon^{m+1}(\tau)\|_{M^{p}_{q}}\Big\|_{\ell^{r}}\nonumber\\&\leq&
C\Big(\|\nabla\upsilon^{m}\|_{L^\infty}\|\upsilon^{m+1}\|_{\dot{N}^{s}_{p,q,r}}+\|\nabla\upsilon^{m+1}\|_{L^\infty}\|\upsilon^{m}\|_{\dot{N}^{s}_{p,q,r}}\Big)
\nonumber\\&\leq&C\|\upsilon^{m}\|_{\dot{N}^{s}_{p,q,r}}\|\upsilon^{m+1}\|_{\dot{N}^{s}_{p,q,r}},\label{R-E33}
\end{eqnarray}
where we have used the Sobolev embedding relations
$\dot{N}^{s-1}_{p,q,r}\hookrightarrow L^{\infty}$ for $s>n/p+1,
1\leq q\leq p<\infty, 1\leq r\leq\infty$ \ or\ $s=n/p+1, 1\leq q\leq
p<\infty$ and $r=1$.

Next we turn our attention to the estimates for the pressure term.
Taking the divergence on both sides of (\ref{R-E26}), we have
\begin{eqnarray}
-\Delta
P^{m+1}=\mathrm{div}(\upsilon^{m}\cdot\nabla)\upsilon^{m+1}\label{R-E34}
\end{eqnarray}
which implies
\begin{eqnarray}
\partial_{i}\partial_{j}P^{m+1}=R_{i}R_{j}\mathrm{div}(\upsilon^{m}\cdot\nabla)\upsilon^{m+1},\label{R-E35}
\end{eqnarray}
where $R_{i}(i=1,2,...,n)$ are the $n$-dimensional Riesz transform.
Since $\mathrm{div}\upsilon^{m}=0$, we obtain
\begin{eqnarray}
\mathrm{div}(\upsilon^{m}\cdot\nabla)\upsilon^{m+1}=\sum^{n}_{k,l=1}\partial_{k}\upsilon^{m}_{l}\partial_{l}\upsilon^{m+1}_{k}=\sum^{n}_{k,l=1}\partial_{l}(\partial_{k}\upsilon^{m}_{l}\upsilon^{m+1}_{k}).\label{R-E36}
\end{eqnarray}
Thus, by Bernstein's lemma, we arrive at
\begin{eqnarray}
\|\nabla P^{m+1}\|_{\dot{N}^s_{p,q,r}}&\leq&
C\sum^{n}_{i,j=1}\|\partial_{i}\partial_{j}P^{m+1}\|_{\dot{N}^{s-1}_{p,q,r}}\nonumber\\&\leq&C\sum^{n}_{i,j,k,l=1}\|R_{i}R_{j}\partial_{k}\upsilon^{m}_{l}\partial_{l}\upsilon^{m+1}_{k}\|_{\dot{N}^{s-1}_{p,q,r}}
\nonumber\\&\leq&C\sum^{n}_{k,l=1}\|\partial_{k}\upsilon^{m}_{l}\partial_{l}\upsilon^{m+1}_{k}\|_{\dot{N}^{s-1}_{p,q,r}}\nonumber\\&\leq&C
\|\nabla\upsilon^{m}\|_{L^\infty}\|\nabla\upsilon^{m+1}\|_{\dot{N}^{s-1}_{p,q,r}}+\|\nabla\upsilon^{m+1}\|_{L^\infty}\|\nabla\upsilon^{m}\|_{\dot{N}^{s-1}_{p,q,r}}
\nonumber\\&\leq&C\|\upsilon^{m}\|_{\dot{N}^{s}_{p,q,r}}\|\upsilon^{m+1}\|_{\dot{N}^{s}_{p,q,r}},\label{R-E37}
\end{eqnarray}
where we have taken $p_{1}=p_{3}=\infty,\ p_{2}=p_{4}=p$ and
$q_{1}=q_{3}=\infty,\ q_{2}=q_{4}=q$ in Lemma \ref{lem3.3}.

It follows from (\ref{R-E32}), (\ref{R-E33}) and (\ref{R-E37}) that
\begin{eqnarray}
\|\upsilon^{m+1}(t)\|_{\dot{N}^s_{p,q,r}}&\leq&
\|\upsilon^{m+1}_{0}\|_{\dot{N}^s_{p,q,r}}+C\int^{t}_{0}\|\upsilon^{m}(\tau)\|_{N^{s}_{p,q,r}}\|\upsilon^{m+1}(\tau)\|_{N^{s}_{p,q,r}}d\tau
\nonumber\\&\leq&C\|\upsilon_{0}\|_{\dot{N}^s_{p,q,r}}+C\int^{t}_{0}\|\upsilon^{m}(\tau)\|_{N^{s}_{p,q,r}}\|\upsilon^{m+1}(\tau)\|_{N^{s}_{p,q,r}}d\tau.
\label{R-E38}
\end{eqnarray}

Moreover, in order to show the estimate in the inhomogeneous
Besov-Morrey spaces $N^s_{p,q,r}$, we need to bound
$\|\upsilon^{m+1}(t)\|_{M^{p}_{q}}$. Similarly, we have
\begin{eqnarray}
|\upsilon^{m+1}(X^{m}(\alpha,t),t)|\leq|\upsilon^{m+1}(\alpha,0)|+\int^{t}_{0}|\nabla
P^{m+1}(X^{m}(\alpha,\tau),\tau)|d\tau.\label{R-E39}
\end{eqnarray}
Furthermore, from Lemma \ref{lem3.1}, we get
\begin{eqnarray}
\|\upsilon^{m+1}(t)\|_{M^{p}_{q}}\leq
\|\upsilon^{m+1}(0)\|_{M^{p}_{q}}+C\int^{t}_{0}\|\nabla
P^{m+1}(\tau)\|_{M^{p}_{q}}d\tau,\label{R-E40}
\end{eqnarray}
where the pressure can be estimate as follows
\begin{eqnarray}
\|\nabla P^{m+1}\|_{M^{p}_{q}}&\leq&
C\sum_{k=1}^{n}\Big\|\nabla(-\Delta)^{-1}\partial_{k}\Big\{(\upsilon^{m}\cdot\nabla)\upsilon_{k}^{m+1}\Big\}\Big\|_{M^{p}_{q}}
\nonumber\\&\leq&
C\|(\upsilon^{m}\cdot\nabla)\upsilon_{k}^{m+1}\|_{M^{p}_{q}}\nonumber\\&\leq&C\|\upsilon^{m}\|_{M^{p}_{q}}\|\nabla\upsilon^{m+1}\|_{L^\infty}\nonumber\\&\leq&
C\|\upsilon^{m}\|_{N^{s}_{p,q,r}}\|\upsilon^{m+1}\|_{N^{s}_{p,q,r}}\label{R-E41}
\end{eqnarray}
Substituting (\ref{R-E41}) into (\ref{R-E40}), we have
\begin{eqnarray}
\|\upsilon^{m+1}(t)\|_{M^{p}_{q}}\leq
C\|\upsilon(0)\|_{M^{p}_{q}}+C\int^{t}_{0}\|\upsilon^{m}(\tau)\|_{N^{s}_{p,q,r}}\|\upsilon^{m+1}(\tau)\|_{N^{s}_{p,q,r}}d\tau.\label{R-E42}
\end{eqnarray}
Therefore, adding (\ref{R-E38}) to (\ref{R-E42}) together, by Lemma
\ref{lem2.3}, we are led to the inhomogeneous space estimate
\begin{eqnarray}
\|\upsilon^{m+1}(t)\|_{N^s_{p,q,r}}\leq
C\|\upsilon_{0}\|_{N^s_{p,q,r}}+C\int^{t}_{0}\|\upsilon^{m}(\tau)\|_{N^{s}_{p,q,r}}\|\upsilon^{m+1}(\tau)\|_{N^{s}_{p,q,r}}d\tau.\label{R-E43}
\end{eqnarray}
It follows from Gronwall's inequality that
\begin{eqnarray}
\|\upsilon^{m+1}(t)\|_{N^s_{p,q,r}}\leq
C\|\upsilon_{0}\|_{N^s_{p,q,r}}\exp\Big(C\int^{t}_{0}\|\upsilon^{m}(\tau)\|_{N^{s}_{p,q,r}}d\tau\Big),\label{R-E44}
\end{eqnarray}
where the generic constant $C>0$ maybe depend on $n$ and $p,q$, but
it is independent of $m$. Therefore we can obtain the uniform
estimates by induction.

In fact, we take $C_{1}>0$ such that
$$\|\upsilon_{0}\|_{N^s_{p,q,r}}\leq \frac{C_{1}}{2C},$$ then the following
inequality holds
\begin{eqnarray}
\|\upsilon^{m}\|_{L^\infty_{T_{1}}(N^s_{p,q,r})}\leq
C_{1},\label{R-E45}
\end{eqnarray}
for all $m\geq0$, provided that $T_{1}>0$ (independent of $m$) is
sufficiently small.

(\ref{R-E45}) can be shown easily by the standard induction. First,
it is true to for $m=0$. Suppose (\ref{R-E45}) holds for $m>0$, it
follows from (\ref{R-E44}) that
\begin{eqnarray}
\|\upsilon^{m+1}(t)\|_{N^s_{p,q,r}}\leq\frac{C_{1}}{2}\exp\Big(C\int^{T}_{0}\|\upsilon^{m}(\tau)\|_{N^{s}_{p,q,r}}d\tau\Big)\leq\frac{C_{1}}{2}\exp(CC_{1}T),
\ \ \mbox{for}\ t\in[0,T].\label{R-E46}
\end{eqnarray}
Hence, (\ref{R-E45}) holds, if we choose $T_{1}>0$ so small that
$\exp(CC_{1}T_{1})\leq2$. Moreover, $T_{1}$ is independent of $m$.\\

Step 3: \textit{\underline{Convergence and existence}}

To prove the convergence, it is sufficient to estimate the
difference of the iteration. Set

$$u^{m+1}=\upsilon^{m+1}-\upsilon^{m},\ \ \ \nabla\Pi^{m+1}=\nabla P^{m+1}-\nabla
P^{m}.$$ Then we take the difference between the equation
(\ref{R-E26}) for the $(m+1)$-th step and the $m$-th step to get

\begin{equation}
\left\{
\begin{array}{l}
\partial_{t}u^{m+1}+(\upsilon^{m}\cdot\nabla)u^{m+1}+(u^{m}\cdot\nabla)\upsilon^{m}+\nabla \Pi^{m+1} = 0,\\
\mathrm{div}u^{m+1}=0,\ \ \mathrm{div}\upsilon^{m}=0,\\
u^{m+1}(x,0)=S_{m+1}\upsilon_{0}(x)-S_{m}\upsilon_{0}(x)=\Delta_{m}\upsilon_{0}(x).
\end{array} \right.\label{R-E47}
\end{equation}

Taking $\dot{\Delta}_{j}(j\in\mathbb{Z})$ on the first equation of
(\ref{R-E47}), we obtain
\begin{equation}
\partial_{t}\dot{\Delta}_{j}u^{m+1}+(\upsilon^{m}\cdot\nabla)\dot{\Delta}_{j}u^{m+1}=[\upsilon^{m}\cdot\nabla,\dot{\Delta}_{j}]u^{m+1}
-\dot{\Delta}_{j}((u^{m}\cdot\nabla)\upsilon^{m})-\dot{\Delta}_{j}\nabla
\Pi^{m+1}.\label{R-E48}
\end{equation}
By the definition of $X^{m}$, similar to (\ref{R-E30}), we arrive at
\begin{eqnarray}
&&|\dot{\Delta}_{j}u^{m+1}(X^{m}(\alpha,t),t)|\nonumber\\&\leq&|\dot{\Delta}_{j}u^{m+1}(\alpha,0)|+\int^{t}_{0}|[\upsilon^{m}\cdot\nabla,\dot{\Delta}_{j}]u^{m+1}(X^{m}(\alpha,\tau),\tau)|d\tau
\nonumber\\&&+\int^{t}_{0}|\dot{\Delta}_{j}((u^{m}\cdot\nabla)\upsilon^{m})(X^{m}(\alpha,\tau),\tau)|d\tau+\int^{t}_{0}|\dot{\Delta}_{j}\nabla
\Pi^{m+1}(X^{m}(\alpha,\tau),\tau)|d\tau. \label{R-E49}
\end{eqnarray}
With the help of Lemma \ref{lem3.1}, we get
\begin{eqnarray}
&&\|\dot{\Delta}_{j}u^{m+1}(t)\|_{M^{p}_{q}}\nonumber\\&\leq&\|\dot{\Delta}_{j}u^{m+1}(0)\|_{M^{p}_{q}}+\int^{t}_{0}\|[\upsilon^{m}\cdot\nabla,\dot{\Delta}_{j}]u^{m+1}(\tau)\|_{M^{p}_{q}}d\tau
\nonumber\\&&+\int^{t}_{0}\|\dot{\Delta}_{j}((u^{m}\cdot\nabla)\upsilon^{m})(\tau)\|_{M^{p}_{q}}d\tau+\int^{t}_{0}\|\dot{\Delta}_{j}\nabla
\Pi^{m+1}(\tau)\|_{M^{p}_{q}}d\tau.\label{R-E50}
\end{eqnarray}
Multiplying both sides by $2^{j(s-1)}$ and taking the $\ell^{r}$
norm, it holds that
\begin{eqnarray}
&&\|u^{m+1}(t)\|_{\dot{N}^{s-1}_{p,q,r}}\nonumber\\&\leq&\|u^{m+1}(0)\|_{\dot{N}^{s-1}_{p,q,r}}
+\int^{t}_{0}\Big\|2^{j(s-1)}\|[\upsilon^{m}\cdot\nabla,\dot{\Delta}_{j}]u^{m+1}(\tau)\|_{M^{p}_{q}}\Big\|_{\ell^{r}}d\tau
\nonumber\\&&+\int^{t}_{0}\|(u^{m}\cdot\nabla)\upsilon^{m}(\tau)\|_{\dot{N}^{s-1}_{p,q,r}}d\tau+\int^{t}_{0}\|\nabla
\Pi^{m+1}(\tau)\|_{\dot{N}^{s-1}_{p,q,r}}d\tau\nonumber\\&=:&I+II+III+IV.
\label{R-E51}
\end{eqnarray}
From Bernstein's inequality, we have
\begin{eqnarray}
I=\|\dot{\Delta}_{m}\upsilon_{0}(x)\|_{\dot{N}^{s-1}_{p,q,r}}\leq
C2^{-m}\|\dot{\Delta}_{m}\upsilon_{0}(x)\|_{\dot{N}^{s}_{p,q,r}}\leq
C2^{-m}\|\upsilon_{0}(x)\|_{\dot{N}^{s}_{p,q,r}}.\label{R-E52}
\end{eqnarray}
For the estimate of $II$, we have by Lemma \ref{lem3.5} (taking
$p_{1}=\infty$ and $q_{1}=q_{2}=q$)
\begin{eqnarray}
II&\leq&
C\int^{t}_{0}\Big(\|\nabla\upsilon^{m}(\tau)\|_{L^\infty}\|u^{m+1}(\tau)\|_{\dot{N}^{s-1}_{p,q,r}}
+\|u^{m+1}(\tau)\|_{L^\infty}\|\upsilon^{m}(\tau)\|_{\dot{N}^{s}_{p,q,r}}\Big)d\tau
\nonumber\\&\leq&
C\int^{t}_{0}\|\upsilon^{m}(\tau)\|_{\dot{N}^{s}_{p,q,r}}\|u^{m+1}(\tau)\|_{\dot{N}^{s-1}_{p,q,r}}d\tau.\label{R-E53}
\end{eqnarray}
For the estimate of $III$, it follows from Lemma \ref{lem3.3} that
\begin{eqnarray}
III&\leq&C\int^{t}_{0}\Big(\|u^{m}(\tau)\|_{L^\infty}\|\nabla\upsilon^{m}(\tau)\|_{\dot{N}^{s-1}_{p,q,r}}
+\|\nabla\upsilon^{m}(\tau)\|_{L^\infty}\|\upsilon^{m}(\tau)\|_{\dot{N}^{s-1}_{p,q,r}}\Big)d\tau
\nonumber\\&\leq&
C\int^{t}_{0}\|u^{m}(\tau)\|_{\dot{N}^{s-1}_{p,q,r}}\|\upsilon^{m}(\tau)\|_{\dot{N}^{s}_{p,q,r}}d\tau.\label{R-E54}
\end{eqnarray}

We can estimate $\nabla\Pi^{m+1}$ as follows. From (\ref{R-E47}), it
follows that
\begin{eqnarray}-\Delta\Pi^{m+1}=\mathrm{div}(\upsilon^{m}\cdot\nabla)u^{m+1}+\mathrm{div}(u^{m}\cdot\nabla)\upsilon^{m},\label{R-E55}\end{eqnarray}
which implies that
\begin{eqnarray}
\partial_{i}\partial_{j}\Pi^{m+1}=R_{i}R_{j}\mathrm{div}(\upsilon^{m}\cdot\nabla)u^{m+1}+R_{i}R_{j}\mathrm{div}(u^{m}\cdot\nabla)\upsilon^{m}.\label{R-E56}
\end{eqnarray}
Thanks to $\mathrm{div}\upsilon^{m}=0$, we have
\begin{eqnarray}
\mathrm{div}(\upsilon^{m}\cdot\nabla)u^{m+1}=\sum_{k,l=1}^{n}\partial_{k}\upsilon^{m}_{l}\partial_{l}u^{m+1}_{k}=\sum_{k,l=1}^{n}\partial_{l}
(\partial_{k}\upsilon^{m}_{l}u^{m+1}_{k}).\label{R-E57}
\end{eqnarray}
Hence, by Bernstein's inequality, it holds that
\begin{eqnarray}
\|\nabla \Pi^{m+1}\|_{\dot{N}^{s-1}_{p,q,r}}&\leq&
C\sum^{n}_{i,j=1}\|\partial_{i}\partial_{j}\Pi^{m+1}\|_{\dot{N}^{s-2}_{p,q,r}}\nonumber\\&\leq&C\sum^{n}_{k,l=1}\|R_{i}R_{j}\partial_{k}\upsilon^{m}_{l}\upsilon^{m+1}_{k}\|_{\dot{N}^{s-1}_{p,q,r}}
+\|(u^{m}\cdot\nabla)\upsilon^{m}\|_{\dot{N}^{s-1}_{p,q,r}}
\nonumber\\&\leq&C
\Big(\|\nabla\upsilon^{m}\|_{L^\infty}\|u^{m+1}\|_{\dot{N}^{s-1}_{p,q,r}}+\|u^{m+1}\|_{L^\infty}\|\nabla\upsilon^{m}\|_{\dot{N}^{s-1}_{p,q,r}}\Big)
\nonumber\\&&+C\Big(\|u^{m}\|_{L^\infty}\|\nabla\upsilon^{m}\|_{\dot{N}^{s-1}_{p,q,r}}+\|\nabla\upsilon^{m}\|_{L^\infty}\|u^{m}\|_{\dot{N}^{s-1}_{p,q,r}}\Big)
\nonumber\\&\leq&C\|\upsilon^{m}\|_{\dot{N}^{s}_{p,q,r}}\|u^{m+1}\|_{\dot{N}^{s-1}_{p,q,r}}+\|\upsilon^{m}\|_{\dot{N}^{s}_{p,q,r}}\|u^{m}\|_{\dot{N}^{s-1}_{p,q,r}},
\label{R-E58}
\end{eqnarray}
where we have taken $p_{1}=p_{3}=\infty,\ p_{2}=p_{4}=p$ and
$q_{1}=q_{3}=\infty,\ q_{2}=q_{4}=q$ in Lemma \ref{lem3.3}.

Furthermore, we have
\begin{eqnarray}
IV&\leq&C\int^{t}_{0}\|\upsilon^{m}(\tau)\|_{\dot{N}^{s}_{p,q,r}}\|u^{m+1}(\tau)\|_{\dot{N}^{s-1}_{p,q,r}}d\tau
+C\int^{t}_{0}\|\upsilon^{m}(\tau)\|_{\dot{N}^{s}_{p,q,r}}\|u^{m}(\tau)\|_{\dot{N}^{s-1}_{p,q,r}}d\tau.\label{R-E59}
\end{eqnarray}
Taking the summation of (\ref{R-E51})-(\ref{R-E54}) and
(\ref{R-E59}), we conclude that
\begin{eqnarray}
&&\|u^{m+1}(t)\|_{\dot{N}^{s-1}_{p,q,r}}\nonumber\\&\leq&C2^{-m}\|\upsilon_{0}(x)\|_{N^{s}_{p,q,r}}+C\int^{t}_{0}\|\upsilon^{m}(\tau)\|_{\dot{N}^{s}_{p,q,r}}\|u^{m+1}(\tau)\|_{\dot{N}^{s-1}_{p,q,r}}d\tau
\nonumber\\&&+C\int^{t}_{0}\|\upsilon^{m}(\tau)\|_{\dot{N}^{s}_{p,q,r}}\|u^{m}(\tau)\|_{\dot{N}^{s-1}_{p,q,r}}d\tau.\label{R-E60}
\end{eqnarray}
Following from the similar procedure of estimate leading to
(\ref{R-E40}), we get
\begin{eqnarray}
\|u^{m+1}(t)\|_{M^{p}_{q}}&\leq&
C\|u^{m+1}(0)\|_{M^{p}_{q}}+C\int^{t}_{0}\|(u^{m}\cdot\nabla)\upsilon^{m}(\tau)\|_{M^{p}_{q}}d\tau\nonumber\\&&+C\int^{t}_{0}\|\nabla\Pi^{m+1}(\tau)\|_{M^{p}_{q}}d\tau,
\label{R-E61}
\end{eqnarray}
where the terms in the right side of (\ref{R-E61}) can be estimated
as
\begin{eqnarray}
\|u^{m+1}(0)\|_{M^{p}_{q}}=\|\dot{\Delta}_{m}\upsilon_{0}\|_{M^{p}_{q}}\leq
C2^{-m}\|\nabla \upsilon_{0}\|_{M^{p}_{q}}\leq C2^{-m}\|
\upsilon_{0}\|_{N^{s}_{p,q,r}},\label{R-E62}
\end{eqnarray}
\begin{eqnarray}
\int^{t}_{0}\|(u^{m}\cdot\nabla)\upsilon^{m}(\tau)\|_{M^{p}_{q}}d\tau&\leq&
C\int^{t}_{0}\|\nabla\upsilon^{m}(\tau)\|_{L^\infty}\|u^{m}(\tau)\|_{M^{p}_{q}}d\tau\nonumber\\&\leq&
C\int^{t}_{0}\|\upsilon^{m}(\tau)\|_{N^{s}_{p,q,r}}\|u^{m}(\tau)\|_{N^{s-1}_{p,q,r}}d\tau\label{R-E63}
\end{eqnarray}
and
\begin{eqnarray}
\|\nabla\Pi^{m+1}\|_{M^{p}_{q}}&\leq&\sum^{n}_{k=1}\|\nabla(-\Delta^{-1})\partial_{k}((u^{m+1}\cdot\nabla)\upsilon^{m}_{k})\|_{M^{p}_{q}}\nonumber\\&&
+\sum^{n}_{k=1}\|\nabla(-\Delta^{-1})\partial_{k}((u^{m}\cdot\nabla)\upsilon^{m}_{k})\|_{M^{p}_{q}}\nonumber\\&\leq
&
C\|(u^{m+1}\cdot\nabla)\upsilon^{m}\|_{M^{p}_{q}}+C\|(u^{m}\cdot\nabla)\upsilon^{m}\|_{M^{p}_{q}}\nonumber\\&\leq&
C\|\nabla\upsilon^{m}\|_{L^\infty}(\|u^{m+1}\|_{M^{p}_{q}}+\|u^{m}\|_{M^{p}_{q}})\nonumber\\&\leq&
C\|\upsilon^{m}\|_{N^{s}_{p,q,r}}(\|u^{m+1}\|_{N^{s-1}_{p,q,r}}+\|u^{m}\|_{N^{s-1}_{p,q,r}}).\label{R-E64}
\end{eqnarray}
Therefore, from (\ref{R-E61})-(\ref{R-E64}), we deduce that
\begin{eqnarray}
&&\|u^{m+1}(t)\|_{M^{p}_{q}}\nonumber\\&\leq&C2^{-m}\|
\upsilon_{0}\|_{N^{s}_{p,q,r}}+C\int^{t}_{0}\|\upsilon^{m}(\tau)\|_{N^{s}_{p,q,r}}(\|u^{m+1}(\tau)\|_{N^{s-1}_{p,q,r}}+\|u^{m}(\tau)\|_{N^{s-1}_{p,q,r}})d\tau.
\label{R-E65}
\end{eqnarray}
Combining (\ref{R-E60}) and (\ref{R-E65}) gives
\begin{eqnarray}
&&\|u^{m+1}(t)\|_{N^{s-1}_{p,q,r}}\nonumber\\&\leq&C2^{-m}\|
\upsilon_{0}\|_{N^{s}_{p,q,r}}+C\int^{t}_{0}\|\upsilon^{m}(\tau)\|_{N^{s}_{p,q,r}}(\|u^{m+1}(\tau)\|_{N^{s-1}_{p,q,r}}+\|u^{m}(\tau)\|_{N^{s-1}_{p,q,r}})d\tau
\label{R-E66}
\end{eqnarray}
for\ $t\in[0,T]$, which yields
\begin{eqnarray}
&&\|u^{m+1}(t)\|_{N^{s-1}_{p,q,r}}\nonumber\\&\leq&CC_{1}2^{-m-1}+CC_{1}T\|u^{m+1}\|_{L^{\infty}_{T_{1}}(N^{s-1}_{p,q,r})}+CC_{1}T\|u^{m}\|_{L^{\infty}_{T_{1}}(N^{s-1}_{p,q,r})},\label{R-E67}
\end{eqnarray}
where $C_{1}$ is the constant obtained for the uniform estimate.
Furthermore, if we choose $T_{1}>0$ sufficiently small so that
$CC_{1}T_{1}\leq1/4$, then
\begin{eqnarray}
&&\|u^{m+1}\|_{L^{\infty}_{T_{1}}(N^{s-1}_{p,q,r})}\nonumber\\&\leq&CC_{1}2^{-m-1}+\frac{1}{2}\|u^{m+1}\|_{L^{\infty}_{T_{1}}(N^{s-1}_{p,q,r})}+\frac{1}{4}\|u^{m}\|_{L^{\infty}_{T_{1}}(N^{s-1}_{p,q,r})},
\label{R-E68}
\end{eqnarray}
which leads to
\begin{eqnarray}
\|u^{m+1}\|_{L^{\infty}_{T_{1}}(N^{s-1}_{p,q,r})}\leq\frac{CC_{1}}{2^{m}},
\   \ m=0,1,2,....\label{R-E69}
\end{eqnarray}
Due to (\ref{R-E69}), it is  clear that
$\|u^{m+1}\|_{L^{\infty}_{T_{1}}(N^{s-1}_{p,q,r})}\rightarrow0$, as
$m$ tends to infinity. Therefore, there exists a limit $\upsilon\in
C([0,T_{1}];N^{s-1}_{p,q,r})$ such that
$\upsilon^{m}(t)\rightarrow\upsilon(t)$ uniformly for
$t\in[0,T_{1}]$ in $N^{s-1}_{p,q,r}$. Moreover, it is easy to see
that $\upsilon$ is a solution of (\ref{R-E1}). Indeed,
$\upsilon\in C([0,T_{1}];N^{s}_{p,q,r})$. This completes the proof of the local existence part.\\

Step 4: \textit{\underline{Uniqueness}}

Suppose that $\upsilon_{1}$ and $\upsilon_{2}$ are two solutions of
(\ref{R-E1}) with the same initial data. Set
$$\delta\upsilon=\upsilon_{1}-\upsilon_{2}.$$ Then we get
\begin{equation}
\left\{
\begin{array}{l}
\partial_{t}\delta\upsilon+(\upsilon_{1}\cdot\nabla)\delta\upsilon+(\delta\upsilon\cdot\nabla)\upsilon_{2}+\nabla \widetilde{\Pi} = 0,\\
\mathrm{div}\upsilon_{1}=0,\ \ \mathrm{div}\upsilon_{2}=0,\\
\upsilon(x,0)=0
\end{array} \right.\label{R-E70}
\end{equation}
where $\widetilde{\Pi}=P_{1}-P_{2}$ with the associated pressures
with $\upsilon_{1}$ and $\upsilon_{2}$, respectively. We follow the
strategy to derive the inequality (\ref{R-E66}) to obtain
\begin{eqnarray}
&&\|\delta\upsilon\|_{L^{\infty}_{T_{1}}(N^{s-1}_{p,q,r})}\nonumber\\&\leq&CT_{1}\|\upsilon_{1}\|_{L^{\infty}_{T_{1}}(N^{s}_{p,q,r})}\|\delta\upsilon\|_{L^{\infty}_{T_{1}}(N^{s-1}_{p,q,r})}
+CT_{1}\|\upsilon_{2}\|_{L^{\infty}_{T_{1}}(N^{s}_{p,q,r})}\|\delta\upsilon\|_{L^{\infty}_{T_{1}}(N^{s-1}_{p,q,r})}
\nonumber\\&\leq&2CC_{1}T_{1}\|\delta\upsilon\|_{L^{\infty}_{T_{1}}(N^{s-1}_{p,q,r})},\label{R-E71}
\end{eqnarray}
where $C_{1}>0$ is the constant obtained by the existence part. So
if we choose $T_{1}>0$ such that $CC_{1}T_{1}\leq1/4$, then
\begin{eqnarray}
\|\delta\upsilon\|_{L^{\infty}_{T_{1}}(N^{s-1}_{p,q,r})}\leq\frac{1}{2}\|\delta\upsilon\|_{L^{\infty}_{T_{1}}(N^{s-1}_{p,q,r})},\label{R-E72}
\end{eqnarray}
which implies $\delta\upsilon=0$ for any $t\in T_{1}$, i.e., $\upsilon_{1}\equiv\upsilon_{2}$ for any $t\in T_{1}$.\\

Step 5: \textit{\underline{Blow-up criterion}}

Suppose that $\upsilon$ is the solution of (\ref{R-E1}) in the class
$C([0,T];N^{s}_{p,q,r})$. As shown by \cite{MB}, for the divergence
free of $\upsilon$, we have the relation between the gradient of
velocity and vorticity
\begin{eqnarray}
\nabla\upsilon=\mathcal{P}(\omega)+A\omega,\label{R-E73}
\end{eqnarray}
where $\mathcal{P}$ is a singular integral operator homogeneous of
degree $-n$ and $A$ is a constant matrix. By the boundedness of the
singular integral operator from $\dot{B}^{0}_{\infty,\infty}$ into
itself \cite{TH}, and Lemma \ref{lem3.2}, we get
\begin{eqnarray}
\|\nabla\upsilon\|_{L^\infty}&\leq&
C\Big(1+\|\nabla\upsilon\|_{\dot{B}^{0}_{\infty,\infty}}(\log^{+}\|\nabla\upsilon\|_{N^{s-1}_{p,q,r}}+1)\Big)
\nonumber\\&\leq&C\Big(1+\|\omega\|_{\dot{B}^{0}_{\infty,\infty}}(\log^{+}\|\upsilon\|_{N^{s}_{p,q,r}}+1)\Big)
\label{R-E74}
\end{eqnarray}
for $s>1+n/p$.

As (\ref{R-E43}) previously, we obtain similarly
\begin{eqnarray}
\|\upsilon(t)\|_{N^{s}_{p,q,r}}\leq
C\|\upsilon_{0}\|_{N^{s}_{p,q,r}}+C\int^{t}_{0}\|\nabla\upsilon(\tau)\|_{L^{\infty}}\|\upsilon(\tau)\|_{N^{s}_{p,q,r}}d\tau.\label{R-E75}
\end{eqnarray}
Substituting (\ref{R-E74}) into (\ref{R-E75}) to get
\begin{eqnarray}
&&\|\upsilon(t)\|_{N^{s}_{p,q,r}}\nonumber\\&\leq&
C\|\upsilon_{0}\|_{N^{s}_{p,q,r}}+C\int^{t}_{0}\Big(1+\|\omega(\tau)\|_{\dot{B}^{0}_{\infty,\infty}}(\log^{+}\|\upsilon(\tau)\|_{N^{s}_{p,q,r}}+1)\Big)\|\upsilon(\tau)\|_{N^{s}_{p,q,r}}d\tau,
\label{R-E76}
\end{eqnarray}
which implies that
\begin{eqnarray}
\|\upsilon(t)\|_{N^{s}_{p,q,r}}\leq
C_{2}\|\upsilon_{0}\|_{N^{s}_{p,q,r}}\exp\Big[C_{3}\exp\Big(C_{4}\int^{t}_{0}(1+\|\omega(\tau)\|_{\dot{B}^{0}_{\infty,\infty}})d\tau\Big)\Big],\label{R-E77}
\end{eqnarray}
by Gronwall's inequality. Here $C_{2},C_{3}$ and $C_{4}$ are some
positive constants. Therefore, if $\limsup_{t\rightarrow
T^{*}-}\|\upsilon(t)\|_{N^{s}_{p,q,r}}=\infty$, then
$\int^{T^{*}}_{0}\|\omega(t)\|_{\dot{B}^{0}_{\infty,\infty}}dt=\infty$.

On the other hand, it follows from Sobolev embedding
$N^{s}_{p,q,r}\hookrightarrow
L^\infty\hookrightarrow\dot{B}^{0}_{\infty,\infty}$ for $s>1+n/p$
that
\begin{eqnarray}
\int^{T^{*}}_{0}\|\omega(t)\|_{\dot{B}^{0}_{\infty,\infty}}dt\leq\int^{T^{*}}_{0}\|\nabla\upsilon(t)\|_{L^\infty}dt\leq
T^{*}\sup_{t\in[0,T^{*}]}\|\upsilon(t)\|_{N^{s}_{p,q,r}}.
\label{R-E78}
\end{eqnarray}
Then
$\int^{T^{*}}_{0}\|\omega(t)\|_{\dot{B}^{0}_{\infty,\infty}}dt=\infty$
implies that $\limsup_{t\rightarrow
T^{*}-}\|\upsilon(t)\|_{N^{s}_{p,q,r}}=\infty$.

Besides, for $s=1+n/p$, since $\dot{B}^{0}_{\infty,1}\hookrightarrow
L^{\infty}$ and the singular integral operator $\mathcal{P}$ is
bounded from $\dot{B}^{0}_{\infty,1}$ into itself, we have
\begin{eqnarray}
\|\nabla\upsilon\|_{L^{\infty}}\leq
C\|\nabla\upsilon\|_{\dot{B}^{0}_{\infty,1}}\leq
C\|\omega\|_{\dot{B}^{0}_{\infty,1}}.\label{R-E79}
\end{eqnarray}
Substituting (\ref{R-E79}) into (\ref{R-E75}), we have
\begin{eqnarray}
\|\upsilon(t)\|_{N^{s}_{p,q,r}}\leq
C\|\upsilon_{0}\|_{N^{s}_{p,q,r}}+C\int^{t}_{0}\|\omega(\tau)\|_{\dot{B}^{0}_{\infty,1}}\|\upsilon(\tau)\|_{N^{s}_{p,q,r}}d\tau.\label{R-E80}
\end{eqnarray}
Then Gronwall's inequality gives
\begin{eqnarray}
\|\upsilon(t)\|_{N^{s}_{p,q,r}}\leq
C_{5}\|\upsilon_{0}\|_{N^{s}_{p,q,r}}\exp\Big(C_{6}\int^{t}_{0}\|\omega(\tau)\|_{\dot{B}^{0}_{\infty,1}}d\tau\Big)\label{R-E81}
\end{eqnarray}
for some positive constants $C_{5}$ and $C_{6}$.

On the other hand, it follows from the Sobolev embedding
$N^{n/p}_{p,q,1}\hookrightarrow\dot{N}^{n/p}_{p,q,1}\hookrightarrow\dot{B}^{0}_{\infty,1}$
that
\begin{eqnarray}
\int^{T}_{0}\|\omega(t)\|_{\dot{B}^{0}_{\infty,1}}dt\leq\int^{T}_{0}\|\nabla\upsilon(t)\|_{N^{n/p}_{p,q,1}}dt\leq
T\sup_{t\in[0,T]}\|\upsilon(t)\|_{N^{1+n/p}_{p,q,1}}.\label{R-E82}
\end{eqnarray}
(\ref{R-E81})-(\ref{R-E82}) implies the blow-up criterion for the case of $s=1+n/p$.\\

Step 6: \textit{\underline{Solve the linear equations}}

To finish the Proof of Theorem \ref{thm1.1}, what left is to solve
the linear equations (\ref{R-E25}). Our idea is to approximate
(\ref{R-E25}) by the linear transport equations. First, we see that
(\ref{R-E25}) is equivalent to the following system
\begin{equation}
\left\{
\begin{array}{l}
\partial_{t}\upsilon+(w\cdot\nabla)\upsilon+\nabla P = 0,\\
-\Delta P=\mathrm{div}((w\cdot\nabla)\upsilon),\\
\upsilon(x,0)=\upsilon_{0}(x), \mathrm{div}\upsilon_{0}=0,
\end{array} \right.\label{R-E83}
\end{equation}
which can be approximated by the following linear transport
equations
\begin{equation}
\left\{
\begin{array}{l}
\partial_{t}\upsilon^{n+1}+(w\cdot\nabla)\upsilon^{n}+\nabla P^{n} = 0,\\
-\Delta P^{n}=\mathrm{div}((w\cdot\nabla)\upsilon^{n}),\\
\upsilon^{n+1}(x,0)=S_{n+1}\upsilon_{0}.
\end{array} \right.\label{R-E84}
\end{equation}
The existence theorem for (\ref{R-E84}) is well-known for each $n$.
To prove the solvability of (\ref{R-E83}), it is suffice to prove
the uniform estimate for the sequence $\{\upsilon^{n+1}\}$ in the
$N^{s}_{p,q,r}$ framework and the Cauchy convergence of the
corresponding sequence. Indeed, This depends on the following a
priori estimates for (\ref{R-E84}), which can be shown in a similar
manner with (\ref{R-E38}) and (\ref{R-E42}). Precisely,
\begin{eqnarray}
\|\upsilon(t)\|_{\dot{N}^{s}_{p,q,r}}&\leq&
C\|\upsilon_{0}\|_{\dot{N}^{s}_{p,q,r}}+C\int^t_{0}\|\nabla
P(\tau)\|_{\dot{N}^{s}_{p,q,r}}d\tau+C
\int^t_{0}\|2^{js}\|[w\cdot\nabla,\dot{\Delta}_{j}]\upsilon(\tau)\|_{M^p_{q}}\|_{\ell^r}d\tau
\nonumber \\&  \leq&
C\|\upsilon_{0}\|_{\dot{N}^{s}_{p,q,r}}+C\int^t_{0}\|w\|_{\dot{N}^{s}_{p,q,r}}\|\upsilon\|_{\dot{N}^{s}_{p,q,r}}d\tau
\label{R-E85}
\end{eqnarray}
and
\begin{eqnarray}
\|\upsilon(t)\|_{M^p_{q}}&\leq&
C\|\upsilon_{0}\|_{M^p_{q}}+C\int^t_{0}\|\nabla
P(\tau)\|_{M^p_{q}}d\tau \nonumber \\&  \leq&
C\|\upsilon_{0}\|_{M^p_{q}}+C\int^t_{0}\|w\|_{N^{s}_{p,q,r}}\|\upsilon\|_{\dot{N}^{s}_{p,q,r}}d\tau.\label{R-E86}
\end{eqnarray}
Hence, we easily arrive at
\begin{eqnarray}
\|\upsilon(t)\|_{N^{s}_{p,q,r}}\leq
C\|\upsilon_{0}\|_{N^{s}_{p,q,r}}+C\int^t_{0}\|w\|_{N^{s}_{p,q,r}}\|\upsilon\|_{N^{s}_{p,q,r}}d\tau.\label{R-E87}
\end{eqnarray}
Applying Gronwall's inequality on (\ref{R-E87}) to get
\begin{eqnarray}
\|\upsilon(t)\|_{N^{s}_{p,q,r}}\leq
C\|\upsilon_{0}\|_{N^{s}_{p,q,r}}\exp\Big(C\int^T_{0}\|w(t)\|_{N^{s}_{p,q,r}}dt\Big),\
\ t\in[0,T].\label{R-E88}
\end{eqnarray}
Having the a priori estimate (\ref{R-E88}), the existence and
uniqueness of solutions for the system (\ref{R-E83}) can be obtained
by the approximate solution sequence $\{\upsilon^{n+1}\}$ of
(\ref{R-E84}). This finished the proof of Proposition \ref{prop4.1}.

Above all, we complete the proof of Theorem \ref{thm1.1} eventually.

\section{Proof of Theorem \ref{thm1.2}}\label{sec:5}
\setcounter{equation}{0}

In the similar spirit, we can prove the Theorem \ref{thm1.2}. In
comparison with the Euler equations (\ref{R-E1}), it is suffice to
handle with the coupling between the velocity field and magnetic
field in the MHD system (\ref{R-E2}). Therefore, we only give the
crucial estimates for conciseness. First, we consider the linear
equations of MHD system:
\begin{equation}
\left\{
\begin{array}{l}
\partial_{t}\upsilon+(w\cdot\nabla)\upsilon-(a\cdot\nabla)b+\nabla\Pi=0
,\\
\partial_{t}b+(w\cdot\nabla)b-(a\cdot\nabla)\upsilon=0,\\
\mathrm{div}\upsilon=0,\ \ \ \mathrm{div}b=0,\\
\upsilon(x,0)=\upsilon_{0}(x),\ \ b(x,0)=b_{0}.
\end{array} \right.\label{R-E89}
\end{equation}
For (\ref{R-E89}), similar to the Proposition \ref{prop4.1}, we have

\begin{prop}\label{prop5.1}
Assume that $\mathrm{div}w=\mathrm{div}a=0, (w,a)\in
L^{\infty}(0,T,N^{s}_{p,q,r})$ for some $T>0, s>1+n/p, 1<q\leq
p<\infty, r\in[1,\infty]$ or $s=1+n/p, 1<q\leq p<\infty$ and $r=1$.
Then for any $(\upsilon_{0},b_{0})\in N^{s}_{p,q,r}$ and
$\mathrm{div}\upsilon_{0}=\mathrm{div}b_{0}=0$, there exists a
unique solution $(\upsilon,b)\in C([0,T];N^{s}_{p,q,r})$ to the
linear system (\ref{R-E89}). And consequently, $\nabla \Pi$ can be
uniquely determined.
\end{prop}

Based on Proposition \ref{prop5.1}, we construct the following
approximate linear system of (\ref{R-E2})
\begin{equation}
\left\{
\begin{array}{l}
\partial_{t}\upsilon^{m+1}+(\upsilon^{m}\cdot\nabla)\upsilon^{m+1}-(b^{m}\cdot\nabla)b^{m+1}+\nabla\Pi^{m+1}=0
,\\
\partial_{t}b^{m+1}+(\upsilon^{m}\cdot\nabla)b^{m+1}-(b^{m}\cdot\nabla)\upsilon^{m+1}=0,\\
\mathrm{div}\upsilon^{m+1}=\mathrm{div}\upsilon^{m}=0,\ \ \ \mathrm{div}b^{m+1}=\mathrm{div}b^{m}=0,\\
\upsilon^{m+1}_{0}=S_{m+1}\upsilon_{0}(x),\ \
b^{m+1}_{0}=S_{m+1}b_{0}.
\end{array} \right.\label{R-E90}
\end{equation}
for $m=0,1,2,...$, where we set $\upsilon^{0}=b^{0}=0$.

In what follows, we give the uniform estimates for approximate
solution sequence $\{(\upsilon^{m+1},b^{m+1})\}$. Indeed, we perform
$\dot{\Delta}_{j}(j\in \mathbb{Z})$ on the first two equations of
(\ref{R-E90}) to get
\begin{equation}
\left\{
\begin{array}{l}
\partial_{t}\dot{\Delta}_{j}\upsilon^{m+1}+(\upsilon^{m}\cdot\nabla)\dot{\Delta}_{j}\upsilon^{m+1}
-(b^{m}\cdot\nabla)\dot{\Delta}_{j}b^{m+1}\\=[\upsilon^{m}\cdot\nabla,\dot{\Delta}_{j}]\upsilon^{m+1}-
[b^{m}\cdot\nabla,\dot{\Delta}_{j}]b^{m+1}-\dot{\Delta}_{j}\nabla
\Pi^{m+1},\\[3mm]
\partial_{t}\dot{\Delta}_{j}b^{m+1}+(\upsilon^{m}\cdot\nabla)\dot{\Delta}_{j}b^{m+1}-(b^{m}\cdot\nabla)\dot{\Delta}_{j}\upsilon^{m+1}\\
=[\upsilon^{m}\cdot\nabla,\dot{\Delta}_{j}]b^{m+1}-[b^{m}\cdot\nabla,\dot{\Delta}_{j}]\upsilon^{m+1}.
\end{array} \right.\label{R-E91}
\end{equation}
To deal with the coupling of $\upsilon^{m+1}$ and $b^{m+1}$, similar
to the particle trajectory mapping $\{X^{m}(\alpha,t)\}$, we define
$\{Y^{m}(\alpha,t)\}$ as follows
\begin{equation}
\left\{
\begin{array}{l}
\partial_{t}Y^{m}(\alpha,t)=(\upsilon^{m}-b^{m})(Y^{m}(\alpha,t),t),\\
Y^{m}(\alpha,0)=\alpha.
\end{array} \right.\label{R-E92}
\end{equation}
Note that $\mathrm{div}(\upsilon^{m}-b^{m})=0$ implies that each
$\alpha\mapsto Y^{m}(\alpha,t)$ is a volume-preserving mapping for
all $t>0$. So it follows from the particle trajectory mapping
(\ref{R-E92}) that
\begin{eqnarray}
\partial_{t}\dot{\Delta}_{j}(\upsilon^{m+1}+b^{m+1})+[(\upsilon^{m}-b^{m})\cdot\nabla]\dot{\Delta}_{j}(\upsilon^{m+1}+b^{m+1})\Big|_{(x,t)
=(Y^{m}(\alpha,t),t)}\nonumber\\
=\frac{\partial}{\partial{t}}\dot{\Delta}_{j}(\upsilon^{m+1}+b^{m+1})(Y^{m}(\alpha,t),t)\label{R-E93}
\end{eqnarray}
which yields
\begin{eqnarray}
&&|\dot{\Delta}_{j}(\upsilon^{m+1}+b^{m+1})(Y^{m}(\alpha,t),t)|\nonumber\\&\leq&|\dot{\Delta}_{j}(\upsilon^{m+1}+b^{m+1})(\alpha,0)|+\int^{t}_{0}|\dot{\Delta}_{j}\nabla
\Pi^{m+1}(Y^{m}(\alpha,\tau),\tau)|d\tau
\nonumber\\&&+\int^{t}_{0}|[\upsilon^{m}\cdot\nabla,\dot{\Delta}_{j}]\upsilon^{m+1}(Y^{m}(\alpha,\tau),\tau)|d\tau
+\int^{t}_{0}|[b^{m}\cdot\nabla,\dot{\Delta}_{j}]b^{m+1}(Y^{m}(\alpha,\tau),\tau)|d\tau\nonumber\\&&+
\int^{t}_{0}|[\upsilon^{m}\cdot\nabla,\dot{\Delta}_{j}]b^{m+1}(Y^{m}(\alpha,\tau),\tau)|d\tau+
\int^{t}_{0}|[b^{m}\cdot\nabla,\dot{\Delta}_{j}]\upsilon^{m+1}(Y^{m}(\alpha,\tau),\tau)|d\tau.\label{R-E94}
\end{eqnarray}
As (\ref{R-E38}), we deduce similarly
\begin{eqnarray}
&&\|\upsilon^{m+1}(t)+b^{m+1}(t)\|_{\dot{N}^s_{p,q,r}}\nonumber\\&\leq&
\|\upsilon^{m+1}_{0}\|_{\dot{N}^s_{p,q,r}}+\|b^{m+1}_{0}\|_{\dot{N}^s_{p,q,r}}
+C\int^{t}_{0}\|\nabla \Pi^{m+1}(\tau)\|_{\dot{N}^s_{p,q,r}}d\tau
\nonumber\\&&+C\int^{t}_{0}\Big\|2^{js}\|[\upsilon^{m}\cdot\nabla,\dot{\Delta}_{j}]\upsilon^{m+1}(\tau)\|_{M^{p}_{q}}\Big\|_{\ell^{r}}d\tau
+C\int^{t}_{0}\Big\|2^{js}\|[b^{m}\cdot\nabla,\dot{\Delta}_{j}]b^{m+1}(\tau)\|_{M^{p}_{q}}\Big\|_{\ell^{r}}d\tau
\nonumber\\&&+C\int^{t}_{0}\Big\|2^{js}\|[\upsilon^{m}\cdot\nabla,\dot{\Delta}_{j}]b^{m+1}\|_{M^{p}_{q}}\Big\|_{\ell^{r}}d\tau
+C\int^{t}_{0}\Big\|2^{js}\|[b^{m}\cdot\nabla,\dot{\Delta}_{j}]\upsilon^{m+1}\|_{M^{p}_{q}}\Big\|_{\ell^{r}}d\tau,
\nonumber\\&\leq&\|\upsilon^{m+1}_{0}\|_{\dot{N}^s_{p,q,r}}+\|b^{m+1}_{0}\|_{\dot{N}^s_{p,q,r}}
+C\int^{t}_{0}\|\nabla \Pi^{m+1}(\tau)\|_{\dot{N}^s_{p,q,r}}d\tau
\nonumber\\&&+C\int^{t}_{0}(\|\upsilon^{m}(\tau)\|_{\dot{N}^s_{p,q,r}}+\|b^{m}(\tau)\|_{\dot{N}^s_{p,q,r}})(\|\upsilon^{m+1}(\tau)\|_{\dot{N}^s_{p,q,r}}+\|b^{m+1}(\tau)\|_{\dot{N}^s_{p,q,r}})d\tau,
\label{R-E95}
\end{eqnarray}
where we have used the commutator estimate in Lemma \ref{lem3.4}.

From (\ref{R-E90}), it follows that
$\Delta\Pi^{m+1}=\mathrm{div}\Big((\upsilon^{m}\cdot\nabla)\upsilon^{m+1}-(b^{m}\cdot\nabla)b^{m+1}\Big),$
which implies
\begin{eqnarray}
\partial_{i}\partial_{j}\Pi^{m+1}=-R_{i}R_{j}\mathrm{div}\Big((\upsilon^{m}\cdot\nabla)\upsilon^{m+1}-(b^{m}\cdot\nabla)b^{m+1}\Big).\label{R-E96}
\end{eqnarray}
Since $\mathrm{div}\upsilon^{m}=\mathrm{div}b^{m}=0$, we have
\begin{eqnarray}
\mathrm{div}(\upsilon^{m}\cdot\nabla)\upsilon^{m+1}=\sum_{k,l=1}^{n}\partial_{k}\upsilon^{m}_{l}\partial_{l}\upsilon^{m+1}_{k}=\sum_{k,l=1}^{n}\partial_{l}
(\partial_{k}\upsilon^{m}_{l}\upsilon^{m+1}_{k})\label{R-E97}
\end{eqnarray}
and
\begin{eqnarray}
\mathrm{div}(b^{m}\cdot\nabla)b^{m+1}=\sum_{k,l=1}^{n}\partial_{k}b^{m}_{l}\partial_{l}b^{m+1}_{k}=\sum_{k,l=1}^{n}\partial_{l}
(\partial_{k}b^{m}_{l}b^{m+1}_{k}).\label{R-E98}
\end{eqnarray}
Hence, by Bernstein's inequality, we have
\begin{eqnarray}
\|\nabla \Pi^{m+1}\|_{\dot{N}^{s}_{p,q,r}}&\leq&
C\sum^{n}_{i,j=1}\|\partial_{i}\partial_{j}\Pi^{m+1}\|_{\dot{N}^{s-1}_{p,q,r}}
\nonumber\\&\leq&C\|\upsilon^{m}\|_{\dot{N}^{s}_{p,q,r}}\|\upsilon^{m+1}\|_{\dot{N}^{s}_{p,q,r}}+\|b^{m}\|_{\dot{N}^{s}_{p,q,r}}\|b^{m+1}\|_{\dot{N}^{s}_{p,q,r}}.\label{R-E99}
\end{eqnarray}
Substitute (\ref{R-E99}) into (\ref{R-E95}) to get
\begin{eqnarray}
&&\|\upsilon^{m+1}(t)+b^{m+1}(t)\|_{\dot{N}^s_{p,q,r}}\nonumber\\&\leq&
\|\upsilon^{m+1}_{0}\|_{\dot{N}^s_{p,q,r}}+\|b^{m+1}_{0}\|_{\dot{N}^s_{p,q,r}}
+C\int^{t}_{0}(\|\upsilon^{m}(\tau)\|_{\dot{N}^s_{p,q,r}}+\|b^{m}(\tau)\|_{\dot{N}^s_{p,q,r}})\nonumber\\
&&\hspace{5mm}\times(\|\upsilon^{m+1}(\tau)\|_{\dot{N}^s_{p,q,r}}+\|b^{m+1}(\tau)\|_{\dot{N}^s_{p,q,r}})d\tau.\label{R-E100}
\end{eqnarray}

Next, we turn to estimate the $M^{p}_{q}$ norm. With the help of the
particle trajectory mapping (\ref{R-E92}), we obtain
\begin{eqnarray}
|(\upsilon^{m+1}+b^{m+1})(X^{m}(\alpha,t),t)|\leq|(\upsilon^{m+1}+b^{m+1})(\alpha,0)|+\int^{t}_{0}|\nabla
\Pi^{m+1}(Y^{m}(\alpha,\tau),\tau)|d\tau.\label{R-E101}
\end{eqnarray}
Furthermore, it follows from the fact
$\det\nabla_{\alpha}Y^{m}(\alpha,t)\equiv1$ that
\begin{eqnarray}
&&\|(\upsilon^{m+1}+b^{m+1})(t)\|_{M^{p}_{q}}\nonumber\\&\leq&
\|\upsilon^{m+1}_{0}+b^{m+1}_{0}\|_{M^{p}_{q}}+C\int^{t}_{0}
\Big(\|\upsilon^{m}\|_{M^{p}_{q}}\|\nabla\upsilon^{m+1}\|_{L^\infty}
+\|b^{m}\|_{M^{p}_{q}}\|\nabla b^{m+1}\|_{L^\infty}\Big) d\tau
\nonumber\\&\leq&\|\upsilon^{m+1}_{0}\|_{M^{p}_{q}}+\|b^{m+1}_{0}\|_{M^{p}_{q}}\nonumber\\&&+C
\int^{t}_{0}
\Big(\|\upsilon^{m}(\tau)\|_{N^s_{p,q,r}}\|\upsilon^{m+1}(\tau)\|_{N^s_{p,q,r}}
+\|b^{m}(\tau)\|_{N^s_{p,q,r}}\|b^{m+1}(\tau)\|_{N^s_{p,q,r}}\Big)
d\tau,\label{R-E102}
\end{eqnarray}
where we have used the Lemma \ref{lem2.5}.

Adding (\ref{R-E102}) to (\ref{R-E100}) together, we arrive at
\begin{eqnarray}
&&\|\upsilon^{m+1}(t)+b^{m+1}(t)\|_{N^s_{p,q,r}}\nonumber\\&\leq&
\|\upsilon_{0}\|_{N^s_{p,q,r}}+\|b_{0}\|_{N^s_{p,q,r}}
+C\int^{t}_{0}(\|\upsilon^{m}(\tau)\|_{N^s_{p,q,r}}+\|b^{m}(\tau)\|_{N^s_{p,q,r}})\nonumber\\
&&\hspace{5mm}\times(\|\upsilon^{m+1}(\tau)\|_{N^s_{p,q,r}}+\|b^{m+1}(\tau)\|_{N^s_{p,q,r}})d\tau.\label{R-E103}
\end{eqnarray}

Besides, We define by $\{Z^{m}(\alpha,t)\}$ the family of particle
trajectory mapping
\begin{equation}
\left\{
\begin{array}{l}
\partial_{t}Z^{m}(\alpha,t)=(\upsilon^{m}+b^{m})(Z^{m}(\alpha,t),t),\\
Z^{m}(\alpha,0)=\alpha.
\end{array} \right.\label{R-E104}
\end{equation}
Note that $\mathrm{div}(\upsilon^{m}+b^{m})=0$ implies that each
$\alpha\mapsto Z^{m}(\alpha,t)$ is a volume-preserving mapping for
all $t>0$. It follows from the particle trajectory mapping
(\ref{R-E104}) that
\begin{eqnarray}
\partial_{t}\dot{\Delta}_{j}(\upsilon^{m+1}-b^{m+1})+[(\upsilon^{m}+b^{m})\cdot\nabla]\dot{\Delta}_{j}(\upsilon^{m+1}-b^{m+1})\Big|_{(x,t)
=(Z^{m}(\alpha,t),t)}\nonumber \\
=\frac{\partial}{\partial{t}}\dot{\Delta}_{j}(\upsilon^{m+1}-b^{m+1})(Z^{m}(\alpha,t),t).\label{R-E105}
\end{eqnarray}
Similar to (\ref{R-E103}), we can deduce that
\begin{eqnarray}
&&\|\upsilon^{m+1}(t)-b^{m+1}(t)\|_{N^s_{p,q,r}}\nonumber\\&\leq&
\|\upsilon_{0}\|_{N^s_{p,q,r}}+\|b_{0}\|_{N^s_{p,q,r}}
+C\int^{t}_{0}(\|\upsilon^{m}(\tau)\|_{N^s_{p,q,r}}+\|b^{m}(\tau)\|_{N^s_{p,q,r}})\nonumber\\
&&\hspace{5mm}\times(\|\upsilon^{m+1}(\tau)\|_{N^s_{p,q,r}}+\|b^{m+1}(\tau)\|_{N^s_{p,q,r}})d\tau.\label{R-E106}
\end{eqnarray}
Together with (\ref{R-E103}) and (\ref{R-E106}), we conclude that
\begin{eqnarray}
&&\|\upsilon^{m+1}(t)\|_{N^s_{p,q,r}}+\|b^{m+1}(t)\|_{N^s_{p,q,r}}\nonumber\\&\leq&
C\|\upsilon_{0}\|_{N^s_{p,q,r}}+C\|b_{0}\|_{N^s_{p,q,r}}
+C\int^{t}_{0}(\|\upsilon^{m}(\tau)\|_{N^s_{p,q,r}}+\|b^{m}(\tau)\|_{N^s_{p,q,r}})\nonumber\\
&&\hspace{5mm}\times(\|\upsilon^{m+1}(\tau)\|_{N^s_{p,q,r}}+\|b^{m+1}(\tau)\|_{N^s_{p,q,r}})d\tau.\label{R-E107}
\end{eqnarray}
By using the Gronwall's inequality, we get
\begin{eqnarray}
&&|\upsilon^{m+1}(t)\|_{N^s_{p,q,r}}+\|b^{m+1}(t)\|_{N^s_{p,q,r}}\nonumber\\&\leq&
C\Big(\|\upsilon_{0}\|_{N^s_{p,q,r}}+\|b_{0}\|_{N^s_{p,q,r}}\Big)\exp\Big\{C\int^{t}_{0}
(\|\upsilon^{m}(\tau)\|_{N^s_{p,q,r}}+\|b^{m}(\tau)\|_{N^s_{p,q,r}})d\tau\Big\}.\label{R-E108}
\end{eqnarray}
Based on the above crucial estimates, we can finish the proof of
Theorem \ref{thm1.2} following from the subsequent steps of the
proof of Theorem \ref{thm1.1}. We would like to skip the details,
for conciseness.

\section*{Acknowledgments}
J. Xu is partially supported by the NSFC (11001127), Special
Foundation of China Postdoctoral Science Foundation (2012T50493),
China Postdoctoral Science Foundation (20110490134) and Postdoctoral
Science Foundation of Jiangsu Province (1102057C).


\begin{thebibliography}{99}
\bibitem {BCD}
H. Bahouri, J.~Y. Chemin and R. Danchin. \textit{Fourier Analysis
and Nonlinear Partial Differential Equations}, Grundlehren der
mathematischen Wissenschaften, Berlin: Springer-Verlag, 2011.

\bibitem{BKM}
J. Beale, T. Kato and A. Majda, Remark on the breakdown of smooth
solutions for the 3-D Euler equations, \textit{Comm. Math. Phys.},
{\bf{94}} (1984) 61-66.

\bibitem{BL}
J. Bergh and J.L\"{o}fstr\"{o}m, \textit{Interpolation spaces. An
introduction}, Springer-Verlag, Berlin-New York, 1976.

\bibitem{C1}
D. Chae, On the wel-posedness of the Euler equtions in the
Trieble-Lizorkin spaces, \textit{Comm. Pure Appl. Math.}, {\bf{55}}
(2002) 654-678.

\bibitem{C2}
D. Chae, Local existence and blow-up criterion for the Euler
equations in the Besov spaces, \textit{Asymptotic Anal.}, {\bf{38}}
(2004) 339-358.

\bibitem{CMZ}
Q.L.Chen, C.X. Miao and Z.F. Zhang, On the well-posedness of the
ideal MHD equations in the Triebel-Lizorkin spaces, \textit{Arch.
Rational Mech. Anal.},  {\bf{195}} (2010) 561-578.

\bibitem{D}
R. Danchin, On the well-posedness of the incompressible
density-dependent Euler equations in the $L^p$ framework, \textit{J.
Diff. Eqs}, {\bf{248}} (2010) 2130-2170.

\bibitem{DF}
R. Danchin and F. Fanelli, The well-posedness issue for the
density-dependent Euler equations in endpoint Besov spaces,
\textit{J. Math. Pures Anal.}, {\bf{96}} (2011) 253-278.

\bibitem{K}
T. Kato, Nonstationary flows of viscous and ideal fluids in
$\mathbb{R}^{3}$, \textit{J. Funct. Anal}, {\bf{9}} (1972) 296-305.

\bibitem{KP}
T. Kato and G. Ponce, Commutator estimates and the Euler and
Navier-Stokes equations, \textit{Comm. Pure Appl. Math.}, {\bf{41}}
(1988) 891-907.

\bibitem{KT}
H. Kozoho and Y.Taniuchi, Limiting case of the Sobolev inequality in
BMO, with applications to the Euler equations,  \textit{Commun.
Math. Phys.}, {\bf{214}} (2000) 191-200.

\bibitem{KOT}
H. Kozoho, T. Ogawa and Y.Taniuchi, The critical Sobolev inequality
in Besov spaces and regularity criterion to some semi-linear
evolution equtions, \textit{Math. Z.}, {\bf{242}} (2002) 251-278.

\bibitem{KY}
H. Kozoho and M. Yamazaki, Semilinear heat equations and the
Navier-Stokes equations with distributions in new funtion spaces as
initial data, \textit{Comm. PDE}, {\bf{19}} (1994) 959-1014.

\bibitem{M}
A. Mazzucato, Besov-Morrey spaces: function spaces theory and
applications to non-linear PDE, \textit{Trans. AMS}, {\bf{355}}
(2002) 1297-1364.

\bibitem{MB}
A. Majda and A.L.Bertozzi, \textit{ Vorticity and incompressible
flow}, Cambridge University Press, Cambridge, 2002.

\bibitem{MY}
C.X.Miao and B.Q.Yuan, Well-posedness of the ideal MHD system in
critical Besov spaces, \textit{Methods Appl. Anal.}, {\bf{13}}
(2006) 89-106.

\bibitem{PP}
H.C.Pak and Y.J.Park, Existence of solution for the Euler equation
in a critical Besov space $B^{1}_{\infty,1}(\mathbb{R}^{n})$,
\textit{Comm. PDE}, {\bf{29}} (2004) 1149-1166.

\bibitem{T}
L. Tang, A remark on the well-posedness of the Euler equation in the
Besov-Morrey space, preprint. \textit{http://www.math.pku.edu.cn:8000/var/preprint/572.pdf}

\bibitem{TH}
H. Triebel, \textit{Theory of function spaces}, Birkh\"{a}user
Verlag, Basel, 1983.

\bibitem{TR}
R. Takada, Local existence and blow-up criterion for the Euler
equations in Besov spaces of weak type, \textit{J. Evol. Equ.},
{\bf{8}} (2008) 693-725.

\bibitem{V1}
M. Vishik, Hydrodynamics in Besov spaces, \textit{Arch. Ration.
Mech. Anal.}, {\bf{145}} (1998) 197-214.

\bibitem{V2}
M. Vishik, Incompressible flows of an ideal fluid with vorticity in
borderline spaces of Besov type, \textit{Ann. Sci. \'{E}cole Norm.
Sup.}, {\bf{32}} (1999) 769-812.

\bibitem{Z}
Y. Zhou, Local well-posedness for the incompressible Euler equations
in the critical Besov spaces, \textit{Ann. Inst. Fourier}, {\bf{54}}
(2004) 773-786.

\bibitem{Z2}
Y. Zhou, Local well-posedness and regularity criterion for the
density-dependent incompressible Euler equations, \textit{Nonlinear
Anal. TMA}, {\bf{73}} (2010) 750-766.

\bibitem{ZXF}
Y. Zhou, Z.P.Xin and J.S.Fan, Well-posedness for the
density-dependent incompressible Euler equations in the critical
Besov spaces (in Chinese), \textit{Sci. Sin. Math}, {\bf{40}} (2010)
959-970.


\end{thebibliography}
\end{document}